\documentclass[10pt]{amsart}
\usepackage{url}
\usepackage{hyperref}
\usepackage{enumerate}
\usepackage{comment}

\makeatletter

\@namedef{subjclassname@2010}{

  \textup{2010} Mathematics Subject Classification}

\usepackage{amssymb, amsmath,amsthm}
\usepackage{mathrsfs}
\newtheorem{thm}{Theorem}[section]

\newtheorem*{thm1}{Theorem}
\newtheorem*{lem1}{Lemma}
\newtheorem{prop}[thm]{Proposition}

\newtheorem{lem}[thm]{Lemma}
\theoremstyle{definition}
\newtheorem{rem}{Remark}

\newtheorem{prob}[thm]{Problem}

\newcommand{\ra}{\rightarrow}
\newcommand{\bk}{\backslash}
\newcommand{\mc}{\mathcal}

\newcommand{\mb}{\mathbb}
\newcommand{\sg}{\sigma}
\newcommand{\eps}{\epsilon}

\renewcommand{\ss}{\substack}

\newcommand{\e}{\varepsilon}

\newcommand{\mbf}{\boldsymbol}

\renewcommand{\bar}{\overline}

\begin{document}
\title{On Shusterman's Goldbach-type problem for sign patterns of the Liouville function}
\author{Alexander P. Mangerel}
\address{Department of Mathematical Sciences, Durham University, Stockton Road, Durham, DH1 3LE, UK}
\email{smangerel@gmail.com}

\begin{abstract}
Let $\lambda$ be the Liouville function. Assuming the Generalised Riemann Hypothesis for Dirichlet $L$-functions (GRH), we show that for \emph{every} sufficiently large even integer $N$ there are $a,b \geq 1$ such that 
$$
a+b = N \text{ and } \lambda(a) = \lambda(b) = -1. 
$$
This conditionally answers an analogue of the binary Goldbach problem for the Liouville function, posed by Shusterman. \\
The latter is a consequence of a quantitative lower bound on the frequency of sign patterns attained by $(\lambda(n),\lambda(N-n))$, for sufficiently large primes $N$. We show, assuming GRH, that there are constants $C,N_0 > 0$ such that for each pattern $(\eta_1,\eta_2) \in \{-1,+1\}^2$ and each prime $N \geq N_0$,
$$
|\{n < N : (\lambda(n),\lambda(N-n)) = (\eta_1,\eta_2)\}| \gg N e^{-C(\log \log N)^{6}}.
$$
The proof makes essential use of the Pierce expansion of rational numbers $n/N$, which may be of interest in other binary problems.
\end{abstract}
\maketitle
\section{Introduction}
\subsection{Main results}
The binary Goldbach conjecture asserts that every even integer $N > 2$ is representable as $N = p+q$, where $p$ and $q$ are prime. Equivalently, for every even $N > 2$ we expect the non-trivial lower bound
\begin{equation}\label{eq:Lambda1correl}
\sum_{1 \leq n < N} 1_{\mb{P}}(n) 1_{\mb{P}}(N-n) > 0
\end{equation}
to hold, writing $\mb{P}$ to denote the set of all primes. This famous claim remains out of reach of modern methods, even assuming the Generalised Riemann Hypothesis (GRH). \\
Inspired by the heuristic relationship between the distribution of primes and the Liouville function $\lambda$, it is natural to consider analogous problems about the behaviour of pairs $(\lambda(n),\lambda(N-n))$, where $1 \leq n < N$. Here, in contrast to the case of primes, there is no parity obstruction to the value distribution of $\lambda(N-n)$, so we need not restrict ourselves to $N$ even. Moreover, there are several different analogous problems that can be posed in this direction. \\
The first concerns obtaining the simplest non-trivial bound on the correlation 
$$
\mc{L}_\lambda(N) := \sum_{1 \leq n < N} \lambda(n) \lambda(N-n),
$$
namely $|\mc{L}_{\lambda}(N)| < N-1$. In our previous paper \cite{Man} we obtained this bound for all $N \geq 11$ by elementary arguments. \\
The next analogous problem to ask is about the \emph{sign patterns} i.e., pairs in $\{-1,+1\}^2$ that are attained by $(\lambda(n),\lambda(N-n))$, for $1 \leq n < N$. The result of \cite{Man} showed that the products $\lambda(n)\lambda(N-n)$ must change sign along $1 \leq n < N$, which implies the existence of integers $1 \leq a,b < N$ such that
$$
\lambda(a) = \lambda(N-a) \text{ and } \lambda(b) = -\lambda(N-b).
$$
Using the involution $b\mapsto N-b$, we find that both the patterns\footnote{Here and elsewhere we abbreviate $+1$ and $-1$ as simply $+$ and $-$.} $(+,-)$ and $(-,+)$ appear among $(\lambda(n),\lambda(N-n))$; however, it is unclear from this which of $(+,+)$ and/or $(-,-)$ is attained by $(\lambda(n),\lambda(N-n))$. \\
In this direction, Shusterman \cite{ShusMO} had earlier asked the following problem.
\begin{prob}[Shusterman] \label{prob:shus}
Is it true that for every \emph{even} integer $N > 2$ there exist positive integers $a,b$ such that $a+b = N$ and $\lambda(a) = \lambda(b) = -1$?
\end{prob}
Shusterman's problem is a weaker version of the binary Goldbach problem, as if $N$ is an even integer such that $N = p+q$ with $p,q$ prime then $\lambda(p) = \lambda(q) = -1$. For that matter, however, we may ask whether the same property holds for \emph{odd} integers $N$ as well, and whether for any such $N$ the sign pattern $(+,+)$, in addition to $(-,-)$, must occur. \\
It is well-known that a resolution of the binary Goldbach problem via sieve methods alone seems to fail due to the \emph{parity problem}, namely that sieves are unable to resolve the parity of the number of prime factors of sifted integers (without implementing additional parity-breaking techniques of the kind introduced by Friedlander and Iwaniec, Harman and others). Naturally, Shusterman's problem offers the same challenge. \\
In this note we address problems like Problem \ref{prob:shus}, conditionally on the GRH for Dirichlet $L$-functions (see Section \ref{subsec:rems} below for a discussion concerning this hypothesis, and how it may be weakened). As mentioned, GRH
is not known to imply the validity of the binary Goldbach conjecture.
\begin{thm}\label{thm:Shus}
Assume GRH. Then there is an effectively computable $N_0$ such that if $N \geq N_0$ is even then we can find $1 \leq a,b < N$ such that $a+b = N$ and $\lambda(a) = \lambda(b) = -1$.
\end{thm}
The latter result is a consequence of the more general Theorem \ref{thm:LGcond} below.
Before stating it we recall that given $y \geq 2$, an integer $n \in \mb{N}$ is called \emph{$y$-friable} (or \emph{$y$-smooth}) if $p|n$ and $p \in \mb{P}$ then $p \leq y$.
\begin{thm} \label{thm:LGcond}
Assume GRH. Let $N$ be a positive integer. Then there is an effectively computable $p_0 \geq 2$ such that either: 
\begin{enumerate}
\item $N$ is $p_0$-friable, or
\item for every $(\eta_1,\eta_2) \in \{-1,+1\}^2$ there is a positive integer $1 \leq a < N$ such that
\begin{align*}
(\lambda(a), \lambda(N-a)) = (\eta_1, \eta_2).
\end{align*}
\end{enumerate}
In particular, case (2) occurs for all sufficiently large\footnote{The word ``squarefree'' here may of course be replaced by ``$k$th power free'', for any $k \geq 2$.} squarefree integers $N$.
\end{thm}
Theorem \ref{thm:LGcond} asserts that each of the sign patterns in $\{-1,+1\}^2$ occurs at least once in the sequence of pairs $(\lambda(n),\lambda(N-n))$, $1 \leq n < N$. This theorem follows from a significantly stronger quantitative assertion on the number of such sign patterns, when $N$ is a sufficiently large \emph{prime}. \\
As mentioned, the main result of \cite{Man} yields $|\mc{L}_{\lambda}(N)| < N-1$.
The proof of Theorem \ref{thm:LGcond} is based on the following key proposition that strengthens this last condition. 
\begin{prop} \label{prop:keyLowerBd}
Assume GRH. Then there is a constant $C > 0$ such that for any prime $N$,
$$
|\mc{L}_{\lambda}(N)| \leq  N - Ne^{-C(\log \log N)^{6}}.
$$
Consequently, for each pair $(\eta_1,\eta_2) \in \{-1,+1\}^2$ there are\footnote{Note that we have not attempted to find the optimal power of $\log\log N$ here, and that the exponent $6$ here could be improved somewhat; however, our argument currently cannot be used to bring it down below $1$.} $\gg  Ne^{-C(\log \log N)^{6}}$ integers $1 \leq n < N$ such that $(\lambda(n) ,\lambda(N-n)) = (\eta_1,\eta_2)$.
\end{prop}
\subsection{Remarks on the main theorems} \label{subsec:rems}
\begin{rem}\label{rem:GRH}
In the proofs of both Theorem \ref{thm:Shus} and Theorem \ref{thm:LGcond} we do not actually need the full force of GRH. In fact, a zero-free region of the shape
$$
\left\{s \in \mb{C} : \text{Re}(s) > 1-\frac{1}{(\log N)^c}, \, |\text{Im}(s)| \leq (\log N)^3\right\},
$$
where $c \in (0,3/50)$, would suffice for our purposes (albeit giving a weaker lower bound for the frequency of each sign pattern in Proposition \ref{prop:keyLowerBd}). See Remark \ref{rem:GRHimp} below for a proof of this.
\end{rem}
\begin{rem}
There has been much recent progress in studying the frequency and distribution of sign patterns of the Liouville function (and indeed other real-valued, non-pretentious multiplicative functions) as a consequence of the work of Matom\"{a}ki, Radziwi\l\l, Tao, Ter\"{a}v\"{a}inen and others on, in particular, correlations of non-pretentious multiplicative functions, see e.g., \cite{Tao}, \cite{TT}, \cite{MR1}, \cite{MRT}, \cite{MRTSigns}, \cite{MRTTZ}.  Consequently, in regards to patterns of pairs such as $(\lambda(n),\lambda(n+h))$, where $h \geq 1$ is \emph{fixed} or \emph{slowly-growing} and $n < N$ is varying, much more is known unconditionally than what we prove here. However, the techniques used towards controlling binary (or higher order) correlations of bounded multiplicative functions with \emph{fixed} or \emph{slowly-growing} shifts $h$ do not seem to be sufficient in order to deal with shifts of size as large\footnote{One difficulty that arises here is that, unlike in the case of fixed shifts $h \geq 1$, where the integers $n$ and $n+h$ are asymptotically of the same scale, the integers $n$ and $N-n$ are of vastly different scales almost everywhere in logarithmic density. Thus, logarithmic averaging seems to offer little benefit here, whereas in contrast it was crucial in obtaining an additional prime variable in which to average over in \cite{Tao} and all of the work that this later inspired.} as $N$, the length of the interval of summation. Thus, the work we do here is not superseded by any of that existing literature. 
\end{rem}
\begin{rem}
The proof shows that $N_0$ in Theorem \ref{thm:Shus} can be computed effectively in terms of $p_0$, the parameter arising in Theorem \ref{thm:LGcond}. Therefore, assuming such a $p_0$ is not too large, there might be a hope of numerically checking all remaining even values $2 < N < N_0$. We leave this matter to the interested reader.
\end{rem}
\begin{rem}
In contrast to Theorem \ref{thm:LGcond}, the evenness of $N$ in Theorem \ref{thm:Shus} makes it simpler to find the sign pattern $(-,-)$, even when $N$ is $p_0$-friable. For example, if $\lambda(N) = +1$ then we immediately have the pattern $(-,-)$ arising from $a = b = N/2$, regardless of whether or not $N$ is $p_0$-friable. \\
In order to extend Theorem \ref{thm:LGcond} to all sufficiently large integers, it would be sufficient to show that for every fixed prime $p$ there is a $k = k_p \geq 1$ such that $p^k$ admits pairs $(a_i,b_i)$, $1 \leq i \leq 4$, yielding each of the patterns in $\{-1,+1\}^2$. In this case, any $N > \prod_{p \leq p_0} p^{k_p}$ would either have a prime factor $p > p_0$, or else a prime power factor $p^\nu$ with $p \leq p_0$ and $\nu > k_p$, and in both cases an appropriate choice of pairs $(a,b)$ for $N$ and each sign pattern $\mbf{\eta} \in \{-1,+1\}^2$ would necessarily exist. We have not as yet been able to prove such a claim.
\end{rem}

\section{Strategy of proof} 
\subsection{Approximate dilation symmetry for Fourier coefficients of $\lambda$}
Let $N$ be a large prime. To prove an upper bound on $|\mc{L}_{\lambda}(N)|$ we obtain lower bounds on both of the expressions $N-1 \pm \mc{L}_{\lambda}(N)$. Observe that for each $\eta \in \{-1,+1\}$,
\begin{equation} \label{eq:deltaUppBd}
N-1 + \eta \mc{L}_{\lambda}(N) = \sum_{n < N} (1 + \eta \lambda(n)\lambda(N-n)) =  2|\{n < N : \lambda(n) \lambda(N-n) = \eta\}|.
\end{equation}
We therefore denote by $\mc{E}(N)$ the set such that 
$$%
|\mc{E}(N)| = \min_{\eta \in \{-1,+1\}} |\{n < N : \lambda(n) = \eta \lambda(N-n)\}|
$$
(breaking ties arbitrarily), and in the sequel we fix $\eta \in \{-1,+1\}$ to be the corresponding sign minimising the right-hand side. In proving Proposition \ref{prop:keyLowerBd} we will relate the size of $\mc{E}(N)$ with\footnote{As usual, given $t \in \mb{R}$ we write $e(t) := e^{2\pi i t}$.}
$$
S_{\lambda}(a/N) := \sum_{n < N} \lambda(n) e(na/N), \quad a \pmod{N},
$$
which are the Fourier coefficients of the map $n \mapsto \lambda(n) 1_{(0,N)}$, defined on $\mb{Z}/N\mb{Z}$. In our previous paper \cite{Man} we observed (see Lemma 3.1 there) that if $|\mc{L}_{\lambda}(N)| = N-1$ then these Fourier coefficients satisfy the dilation property
\begin{equation}\label{eq:GaussSumLike}
S_{\lambda}(ar/N) = \lambda(r) S_{\lambda}(a/N) \text{ for all } 1 \leq R < N,
\end{equation}
reminiscent of the symmetries of the Gauss sums of the Legendre symbol modulo $N$. This rigidity helped us obtain (by Fourier inversion $\pmod{N}$) a contradiction, as $\lambda$ does not behave like\footnote{The precise incarnation of this admittedly vague statement that is used in \cite{Man} is that if $N>5$ is prime then the Legendre symbol mod $N$ has a \emph{prime quadratic residue} $q < N$, whereas $\lambda(q) = -1$.} a Dirichlet character. \\
In the same direction, a crucial input to the proof of Proposition \ref{prop:keyLowerBd} is the following approximate version of \eqref{eq:GaussSumLike}, assuming some information about $\mc{E}(N)$.
\begin{prop} \label{prop:approxDil}
Let $3 \leq q \leq \sqrt{\log N}$ and let $T \leq e^{\tfrac{1}{40}q^{1/5}}$. Then for every integer $1 \leq r \leq T$,
$$
\frac{1}{N}\sum_{a \pmod{N}} \left|S_{\lambda}(ra/N) -\lambda(r)S_{\lambda}(a/N)\right|^2 \ll 2^{(1+o(1)) q^2} |\mc{E}(N)| + NT^{-1+o(1)}.
$$
\end{prop}
See Section \ref{subsec:prop21} below for a discussion of the proof of Proposition \ref{prop:approxDil}.
\begin{rem}
The choice of $q$ required to make this bound $o(N)$ is of course restricted by the size of $|\mc{E}(N)|$. In principle, as long as $|\mc{E}(N)| = o(N)$ a suitable choice of $q = q(N) \ra \infty$ is available. However, for our application we will require the above estimate to hold with $T \geq (\log N)^{2+\e}$, which therefore forbids us from choosing $q$ to grow too slowly. Consequently, our method does not allow us to show that $|\mc{L}_{\lambda}(N)| \leq (1-c)N$ for some fixed $c > 0$.
\end{rem}
In contrast to the above proposition, we will show, assuming GRH, that the dilation symmetry property fails for at least some primes $p$ bigger than $(\log N)^{2+\e}$.
\begin{prop}\label{prop:GRHavg}
Assume GRH. Then for any $\e > 0$ and any $P \geq (\log N)^2(\log\log N)^{2+\e}$ there exists a prime $p \in (P,2P]$ such that 
$$
\frac{1}{N}\sum_{a\pmod{N}} \left|S_{\lambda}(pa/N) - \lambda(p)S_{\lambda}(a/N)\right|^2 \gg N.
$$
\end{prop}
The proof of this proposition involves a straightforward application of orthogonality of characters modulo $N$ together with a standard consequence of GRH to bounds on character sums over primes. It will be given in Section \ref{sec:GRH}.
\subsection{Towards a proof of Proposition \ref{prop:approxDil}} \label{subsec:prop21}
Let us briefly explain the proof strategy for Proposition \ref{prop:approxDil}, which is at the heart of this paper. To do this it will be helpful to introduce some notation which we will use throughout. In the sequel, given integers $d \geq 1$ and $1 \leq n < N$ we write 
$$
\phi_d(n) := N \{dn/N\}, 
\quad \Lambda_d(n) := \lambda(dn)\lambda(\phi_d(n)) \in \{-1,+1\}.
$$ 
Also, write
$$
\mc{E}_d(N) := \{n < N : \Lambda_d(n) = -1\}.
$$
By construction, $\phi_d$ maps $\mb{Z} \cap (0,N)$ to itself in such a way that $\phi_d(n) \equiv dn \pmod{N}$. We note moreover that if $n < N/d$ then $\phi_d(n) = dn$, and thus $\Lambda_d(n) = +1$. In particular, 
\begin{equation}\label{eq:Edinit}
\mc{E}_d(N) \cap (0,N/d) = \emptyset \text{ for all } d \geq 1,
\end{equation}
a fact we will use often in the sequel. \\
There are two key ideas that enter the picture. 
\subsubsection{A recursive bound for $|\mc{E}_d(N)|$}
The first idea is inspired by the work in \cite[Sec. 3]{Man}. There, we deduced the dilation property \eqref{eq:GaussSumLike} under the assumption that $\lambda(n) = \eps \lambda(N-n)$ for \emph{all} $1 \leq n <N$, for a fixed choice $\eps \in \{-1,+1\}$. The argument was inductive in nature, effectively showing that if for every $1 \leq r < d$ we have
$$
S_{\lambda}(ra/N) = \lambda(r) S_{\lambda}(a/N) \text{ for all } a \pmod{N}
$$ 
then the same is true for $r = d$. By reappraising the argument in \cite{Man}, one can show (see Section \ref{subsec:Pierce}) that when $d = p$ is prime, if $\lambda(n) = \eps\lambda(N-n)$ for \emph{all} $n<N$ then
$$
\Lambda_p(n) = \Lambda_r(\phi_p(n)) \Lambda_p(n\{N/n\}),
$$
where $r := \left \lfloor N/n\right\rfloor$. If $r \geq p$ then $\Lambda_p(n) = +1$ because of \eqref{eq:Edinit}; otherwise $1 \leq r < p$ and we can pass from computing $\Lambda_p(n)$ to computing a value of $\Lambda_r$ and a value of $\Lambda_p$ with the smaller argument $n\{N/n\} < n$. Moreover, this identity may be iterated to express $\Lambda_p(n\{N/n\})$ similarly. \\
In our case, however, we only assume that there is $\eta \in \{-1,+1\}$ such that $\mc{E}(N) = \{n < N : \lambda(n) \neq -\eta\lambda(N-n)\}$ is \emph{sparse}, rather than empty. By incorporating the possibility of exceptions to this process, we show upon iteration that 
$$
\Lambda_p(n) = \prod_{1 \leq j \leq k_p(n)} \Lambda_{r_j(n)}(n_{j-1})
$$
for all but $\leq 2p|\mc{E}(N)|$ exceptions, where $1 \leq k_p(n) < p$, $1 \leq r_1(n) < r_2(n) < \cdots < r_{k_p(n)}(n) < p$, and the $n_j$ are all positive integers determined explicitly by $n$ (see Lemma \ref{lem:recurrence}). The integers $r_j(n)$ and $n_j$ are linked to what is called the \emph{Pierce expansion} of $n/N$, see Section \ref{subsec:Pierce} for further details. \\
Using this process, we obtain a recursive inequality relating $|\mc{E}_d(N)|$ for prime $d$, to $|\mc{E}_r(N)|$ with $1 \leq r < d$ (the case of composite $d$ is simpler and may ultimately be reduced to the prime case anyway; see Lemma \ref{lem:composite}). This leads to the bound $|\mc{E}_d(N)| \leq 2^{d^2}|\mc{E}(N)|$ for all $d \geq 2$, as well as the alternative bound
\begin{equation}\label{eq:friableEd}
|\mc{E}_d(N)| \leq (\log d)^2 2^{q^2} |\mc{E}(N)| \text{ for all } d \geq 2 \text{ such that } P^+(d) \leq q
\end{equation}
that improves upon the first when $P^+(d)$ is somewhat smaller than $d$ (see Proposition \ref{prop:recur}). The significant advantage of these estimates is that they do not involve any additional error terms of the shape $o(N)$. As a consequence, we are able to gain stronger conclusions assuming stronger hypotheses on the sparseness of $\mc{E}(N)$. \\
A better understanding of Pierce expansions would allow improvements to \eqref{eq:friableEd} in the $q$-aspect; at present we are forced to take $q = O(\sqrt{\log N})$. In particular, this argument alone is insufficient for us to access primes in the the GRH range $p \geq (\log N)^{2+\e}$, and thus to obtain a contradiction to Lemma \ref{prop:GRHavg}. See Remark \ref{rem:nurprob} below for a discussion of this bound, what we expect might be true, and how the heuristic truth might allow us to obtain an unconditional result.
\subsubsection{Relations for $\mc{E}_d(N)$ using multiplicativity}
To obtain bounds for $|\mc{E}_d(N)|$ when $d$ is non-friable we exploit the multiplicative nature of the definition of $\mc{E}_d(N)$. For any $a,b \geq 2$ coprime to $N$ we show, by relating them to $\mc{E}_{ab}(N)$, that\footnote{Given sets $S,T$ we denote by $S \triangle T$ the symmetric difference $S \triangle T = (S \cup T) \bk (S \cap T) = (S\bk T) \cup (T \bk S)$.}
$$
\mc{E}_a(N) \triangle \phi_a^{-1} (\mc{E}_b(N)) = \mc{E}_b(N) \triangle \phi_b^{-1} (\mc{E}_a(N)),
$$
see Lemma \ref{lem:shiftEb}. This establishes a sort of ``reciprocity'' between exceptional sets $\mc{E}_b(N)$ and $\mc{E}_a(N)$.  By elementary manipulations, we deduce  that
$$
|\mc{E}_b(N) \triangle \phi_a^{-1} (\mc{E}_b(N))| \leq 2|\mc{E}_a(N)|.
$$
When applied with $q$-friable values of $a$ (so that by \eqref{eq:friableEd}, $|\mc{E}_a(N)|$ is small), we obtain that for any $1 \leq b < \sqrt{N}$ the set $\mc{E}_b(N)$ is left essentially invariant under multiplication by $a \pmod{N}$.  \\
This observation lends itself to controlling the Fourier coefficients $\pmod{N}$ of the set $\mc{E}_b(N)$ (or more precisely of its indicator function), which we do in Section \ref{sec:discrep}. Indeed, we show that if $T \leq e^{\tfrac{1}{40}q^{1/5}}$ and $1 \leq k < N$ then
$$
\sum_{n \in \mc{E}_b(N)} e(kn/N) \approx \sum_{n \in \mc{E}_b(N)} \left(\frac{1}{\Psi(T,q)} \sum_{\ss{a \leq T \\ P^+(a) \leq q}} e(kan/N)\right),
$$
where $\Psi(T,q) := |\{n \leq T : P^+(a) \leq q\}|$. Using bounds for exponential sums over friable numbers due to Harper \cite{Har} and to Fouvry-Tenenbaum \cite{FoTe}, we obtain upper bounds for these Fourier coefficients, uniformly over $1 \leq k < N$. \\ Consequently, if $|\mc{E}_b(N)|$ is \emph{not} small then it must be \emph{uniformly distributed} in $N$, in the sense for any interval $I \subset (0,1)$, writing its length as $|I|$,
$$
|\{n \in \mc{E}_b(N): \, n/N \in I \}| \sim |\mc{E}_b(N)| \cdot |I|.
$$ 
As $\mc{E}_b(N) \cap (0,N/b) = \emptyset$ by \eqref{eq:Edinit}, this\footnote{This conclusion can only be obtained in a restricted range of $b$, as otherwise the quantity $|\mc{E}_b(N)||I| = \tfrac{|\mc{E}_b(N)|}{b}$ is smaller than the error term in our estimates.} rules out the latter possibility, and we hence establish that $|\mc{E}_b(N)|$ is small. Arguing precisely along the above lines allows us to establish Proposition \ref{prop:approxDil}.  
\section{Proof of Proposition \ref{prop:GRHavg}} \label{sec:GRH}
Our first task is to prove Proposition \ref{prop:GRHavg}. 
The following easy lemma relates $\mc{E}_d(N)$ to the Fourier coefficients $S_{\lambda}(a/N)$.
\begin{lem} \label{lem:reltoEd}
Let $N,d \geq 2$ with $(d,N) = 1$. Then
$$
\frac{1}{N}\sum_{a \pmod{N}} \left|S_{\lambda}(da/N)- \lambda(d) S_{\lambda}(a/N)\right|^2 = 4|\mc{E}_d(N)|.
$$
\end{lem}
\begin{proof}
Expanding the square and employing Plancherel's theorem $\pmod{N}$ (noting that $a \mapsto ad \pmod{N}$ is a bijection), the left-hand side is
\begin{align*}
&\frac{1}{N} \sum_{a \pmod{N}} |S_{\lambda}(da/N)|^2 +\frac{1}{N} \sum_{a \pmod{N}} |S_{\lambda}(a/N)|^2 - \frac{2\lambda(d)}{N}\text{Re}\left(\sum_{a \pmod{N}} S_{\lambda}(da/N) \bar{S}_{\lambda}(a/N)\right) \\
&= 2\left(N-1 - \frac{\lambda(d)}{N}\text{Re}\left(\sum_{a \pmod{N}} S_{\lambda}(da/N) \bar{S}_{\lambda}(a/N)\right)\right).
\end{align*}
We also have
\begin{equation} \label{eq:periodEqn}
\frac{\lambda(d)}{N}\sum_{a \pmod{N}} S_{\lambda}(da/N) \bar{S}_{\lambda}(a/N) = \sum_{n,m < N} \lambda(dn)\lambda(m) \frac{1}{N} \sum_{a \pmod{N}} e(a(dn-m)/N) = \sum_{\ss{n,m < N \\ dn \equiv m \pmod{N}}} \lambda(dn)\lambda(m).
\end{equation}
Observe that $nd \equiv m \pmod{N}$ with $1 \leq n,m < N$ if, and only if, $m = N\{nd/N\} = \phi_d(n)$, whence
$$
\lambda(nd)\lambda(m) = \lambda(nd)\lambda(\phi_d(n)) = \Lambda_d(n).
$$
Thus, we have
\begin{align*}
&\frac{1}{N}\sum_{a \pmod{N}} \left|S_{\lambda}(da/N)- \lambda(d) S_{\lambda}(a/N)\right|^2 = 2(N-1 - \sum_{n < N} \Lambda_d(n)) \\
&= 2(N-1 - (N-1 - 2|\{n < N : \Lambda_d(n) = -1\}|)) = 4|\mc{E}_d(N)|,
\end{align*}
as claimed.
\end{proof}
As a consequence of Lemma \ref{lem:reltoEd}, it is clear that Proposition \ref{prop:approxDil} yields an upper bound for $|\mc{E}_d(N)|$ for $d$ in a wide range. Our contradiction will arise from a lower bound for $|\mc{E}_p(N)|$, on average over primes $p \in (P,2P]$ in a range of $P = T$ that is accessible in Proposition \ref{prop:approxDil}. To this end, we have the following standard result. 
\begin{lem}\label{lem:withZFR}
Let $3 \leq T \leq N$. Let $\chi$ be a non-principal character modulo $N$ whose Dirichlet $L$-function $L(s,\chi)$ is zero-free in the region 
$$
\{s \in \mb{C} : \alpha < \text{Re}(s) \leq 1, \, |\text{Im}(s)| \leq T\}.
$$
Then for any\footnote{As is standard, the summation condition $p \sim P$ means $P < p \leq 2P$.} $3 \leq P < N/2$,
$$
\frac{1}{\pi(2P)-\pi(P)} \left|\sum_{p \sim P} \chi(p)\right| \ll \frac{1}{\log P}+ \frac{(\log N)^2}{T} + \frac{(\log N)(\log T)}{P^{1-\alpha}}.
$$
\end{lem}
\begin{proof}
The argument is similar to that of \cite[Prop. 2.9]{Pil}. First, by the prime number theorem we have
\begin{align}
\frac{1}{\pi(2P)-\pi(P)} \left|\sum_{p \sim P} \chi(p)\right| &\ll \frac{\log P}{P} \left|\sum_{p \sim P} \chi(p)\right| = \frac{1}{P} \left|\sum_{p \sim P} \chi(p) \log p\right| + O\left(\frac{1}{\log P}\right) \nonumber\\
&= \frac{1}{P} \left|\sum_{n \sim P} \chi(n)\Lambda(n)\right| + O\left(\frac{1}{\log P}\right). \label{eq:passToLambda}
\end{align}
Using the Hadamard product of $L(s,\chi)$ and the Perron formula, one can show (see e.g. \cite[Thm. 11.3]{Kou})
$$
\frac{1}{P}\sum_{n \sim P} \chi(n)\Lambda(n) = -\sum_{\ss{\rho = \beta + i\gamma \\ L(\rho,\chi) = 0 \\ |\gamma| \leq T}} \frac{P^{\rho-1}(2^{\rho}-1)}{\rho} + O\left(\frac{(\log N)^2}{T}\right).
$$
As for any unit subinterval $I \subset [-T,T]$ there are $O(\log N)$ zeros $\rho$ with $\gamma \in I$ \cite[Thm. 10.17]{MonVau}, we have
\begin{align*}
\frac{1}{P}\sum_{n \sim P} \chi(n)\Lambda(n) &\ll\sum_{0 \leq k \leq T} \sum_{\ss{\rho = \beta + i\gamma \\ L(\rho,\chi) = 0 \\ |\gamma-k| \leq 1/2}} \left|\frac{P^{\rho-1}(2^{\rho}-1)}{\rho}\right| + O\left(\frac{(\log N)^2}{T}\right) \ll (\log N) P^{\alpha-1} \sum_{0 \leq k \leq T} \frac{1}{1+k} + \frac{(\log N)^2}{T} \\
&\ll \frac{(\log N) (\log T)}{P^{1-\alpha}} + \frac{(\log N)^2}{T}.
\end{align*}
Combining this estimate with \eqref{eq:passToLambda} yields the claim.
\end{proof}
Assuming some information on a zero-free region for the product $\prod_{\ss{\chi \pmod{N} \\ \chi \neq \chi_0}} L(s,\chi)$
furnishes the following.
\begin{lem} \label{lem:GRHEp}
Assume that the product $L$-function $\prod_{\ss{\chi \pmod{N} \\ \chi \neq \chi_0}} L(s,\chi)$ is zero-free in the region
$$
\{s \in \mb{C} : \alpha < \text{Re}(s) \leq 1, \, |\text{Im}(s)| \leq (\log N)^3\}.
$$
Then for $3 \leq P < N/2$,
$$
\frac{1}{\pi(2P)-\pi(P)} \sum_{p \sim P} |\mc{E}_p(N)| = \left(\frac{1}{2} + O\left(\frac{(\log N)(\log \log N)}{P^{1-\alpha}} + \frac{1}{\log P}\right)\right)N.
$$
\end{lem}
\begin{proof}
From the proof of Lemma \ref{lem:reltoEd} and in particular \eqref{eq:periodEqn}, we have
$$
|\mc{E}_p(N)| = \frac{1}{2}\left(N-1 - \sum_{\ss{m,n < N \\ np \equiv m \pmod{N}}} \lambda(pn)\lambda(m)\right).
$$
For convenience, write $\tilde{\pi}(P) := \pi(2P)-\pi(P)$. Averaging over the primes $p \sim P$ and using orthogonality of characters modulo $N$, we get
\begin{align*}
\frac{1}{\tilde{\pi}(P)} \sum_{p \sim P} |\mc{E}_p(N)| &= \frac{1}{2}(N-1) - \frac{1}{\tilde{\pi}(P)} \sum_{p \sim P} \sum_{m,n < N} \lambda(pn)\lambda(m) 1_{pn \equiv m\pmod{N}} \\ 
&= \frac{1}{2}(N-1) - \frac{1}{N-1} \sum_{\chi \pmod{N}} \left(\frac{1}{\tilde{\pi}(P)} \sum_{p \sim P} \chi(p)\lambda(p)\right) \left|\sum_{n < N} \lambda(n) \chi(n)\right|^2.
\end{align*}
By the prime number theorem, the contribution from the principal character is 
$$
\frac{1}{N-1}\left|\sum_{n < N} \lambda(n)\right|^2 \ll Ne^{-2\sqrt{\log N}}.
$$
Furthermore, by Lemma \ref{lem:withZFR} we have, for every $\chi \neq \chi_0$,
\begin{align} \label{eq:GRHprime}
\frac{1}{\tilde{\pi}(P)} \sum_{p \sim P} \chi(p) \ll \frac{1}{\log P} + \frac{(\log N) (\log \log N)}{P^{1-\alpha}},
\end{align}
and by orthogonality of characters we obtain
$$
\frac{1}{N-1} \sum_{\chi \neq \chi_0} \left|\sum_{n < N} \lambda(n)\chi(n)\right|^2 \leq N-1.
$$
As $P < N$, the claim follows on combining these bounds. 
\end{proof}
\begin{proof}[Proof of Proposition \ref{prop:GRHavg}]
Assuming GRH, we may apply the previous lemma with $\alpha = 1/2$ to get that if $P \geq (\log N)^2(\log\log N)^{2+\e}$ then
$$
\max_{P < p \leq 2P} |\mc{E}_p(N)| \geq \frac{1}{\tilde{\pi}(P)} \sum_{p \sim P} |\mc{E}_p(N)| = \left(\frac{1}{2} + o_{\e}(1)\right) N.
$$
Combining this with Lemma \ref{lem:reltoEd}, the claim follows immediately.
\end{proof}

\begin{rem}
In the statement of Lemma \ref{lem:GRHEp} we assumed a uniformly large zero-free region for \emph{all} of the non-principal Dirichlet characters modulo $N$, in order to obtain a uniform bound for all prime-supported character sums.
It is natural to ask whether one could avoid any such assumption using zero-density estimates, and isolating the contribution of the small set $\mc{S}$ of characters that lack the appropriate zero-free region. Given current knowledge about Dirichlet $L$-functions, it cannot be ruled out that there are characters $\pmod{N}$ whose prime sums over an interval $[P,2P]$ with $P = N^{o(1)}$ yields no cancellation. Therefore, in order to obtain a suitable lower bound we must obtain an estimate like
$$
\frac{1}{N}\sum_{\chi \in \mc{S}} \left|\sum_{n < N} \lambda(n)\chi(n)\right|^2 = o(N).
$$
By H\"{o}lder's inequality, it would suffice to get a bound of the shape 
$$
\frac{1}{N}\sum_{\ss{\chi  \pmod{N} \\ \chi \neq \chi_0}} \left|\sum_{n < N} \lambda(n) \chi(n)\right|^p \ll N^{p/2}(\log N)^{O(1)}
$$
for \emph{some} $p >2$, but such a bound is not\footnote{See \cite{Har2} for corresponding results with $p \leq 2$; in contrast, when $p > 2$, the contribution from the principal character outweighs the rest of the sum over $\chi$ when it is included, making appeals to orthogonality less direct.} obvious.\\
It should be noted that if $N$ is an exceptional modulus with exceptional character $\tilde{\chi}$ then $\tilde{\chi}(n) = \lambda(n)$ for many $n < N$, whence it contributes $\gg N^{p-1} \gg N^{p/2}$ to the $p$th moment when $p > 2$. While, as previously mentioned our main theorems are already true for exceptional moduli, it it not clear to the author that characters whose $L$-functions have even a standard zero-free region cannot also conspire to give bounds of this size (and only a small number of these characters are tolerable).
\end{rem}

\section{A recursive bound for $|\mc{E}_d(N)|$} \label{sec:recur}
Let $N$ be a large integer (which we will eventually take to be prime). To prove Proposition \ref{prop:approxDil}, we must bound $|\mc{E}_d(N)|$, the number of exceptions $n < N$ to $\Lambda_d(n) = +1$. We think of these as \emph{exceptional} because of their relationship with the set $\mc{E}(N)$ of exceptions to $\lambda(n)\lambda(N-n) = -\eta$, which we will now explain. In preparation for this, let us define
$$
g(d) := |\mc{E}_d(N)|/|\mc{E}(N)|, \quad d \geq 1
$$
(for a reminder of the notation, see the beginning of Section \ref{subsec:prop21}). Note that this is well-defined for $N$ sufficiently large, since the main theorem in \cite{Man} implies that $\mc{E}(N) \neq \emptyset$ for any $N \geq 11$. \\
In this section we shall prove the following recursive bound on $g(d)$, which will be suitable for $q$-friable values of $d$.
\begin{prop} \label{prop:recur}
Let $N$ be a large prime. For every integer $2 \leq d < N$ we have $g(d) \leq 2^{d^2}$. Moreover,  if $P^+(d) \leq q$, then 
$$
g(d) \ll (\log d)^22^{q^2}.
$$
\end{prop}
In the next section, we will see how this result may be used to bound the sizes of $\mc{E}_d(N)$, for non-friable valued of $d$ as well.
\subsection{Properties of the maps $\phi_d$}
We begin with two simple but useful observations about the maps $\phi_d(n) := N\{nd/N\}$, which are self-maps on the interval $\mb{Z} \cap (0,N)$. (For the moment we allow any choice of large $N$; starting in Section \ref{subsec:Pierce} we will focus on $N$ prime.)
\begin{lem} \label{lem:phidInj}
For each $d \geq 1$ that is coprime to $N$, the map $\phi_d$ is bijective on $\mb{Z} \cap (0,N)$.
\end{lem}
\begin{proof}
If $0 < n_1,n_2 < N$ then
$$
\phi_d(n_1) = \phi_d(n_2) \Leftrightarrow \left\{\frac{dn_1}{N}\right\} = \left\{\frac{dn_2}{N}\right\} \Leftrightarrow dn_1 \equiv dn_2 \pmod{N} \Leftrightarrow n_1 \equiv n_2 \pmod{N} \Leftrightarrow n_1 = n_2,
$$
which proves injectivity. \\
Next, suppose $m \in \mb{Z} \cap (0,N)$. As $N \nmid m$ and $(d,N) = 1$, we can find $0 < n  <N$ such that $dn \equiv m \pmod{N}$. Hence, $m = N\{dn/N\} = \phi_d(n)$, and $\phi_d$ is surjective.
\end{proof}
Below, as usual given sets $S,T$ we write their symmetric difference as 
$$
S \triangle T := (S \cup T) \bk (S\cap T) = (S \bk T) \cup (T \bk S). 
$$
The following lemma, which will be used in an essential way in Section \ref{sec:discrep}, establishes a relationship between the sizes of apparently unrelated sets $\mc{E}_a(N)$ and $\mc{E}_b(N)$, whenever $a,b \geq 2$ are both coprime to $N$. We will use it, in tandem with Proposition \ref{prop:recur}, to obtain information about $\mc{E}_b(N)$ with non-friable $b$ from information about $\mc{E}_a(N)$ with friable $a$.
\begin{lem} \label{lem:shiftEb}
Let $a,b \geq 2$ with $(ab,N) = 1$. Then
$$
|\mc{E}_b(N) \triangle \phi_a^{-1}(\mc{E}_b(N))| \leq 2|\mc{E}_a(N)|.
$$
\end{lem}
\begin{proof}
Let $1 \leq n < N$. Observe the identity
$$
\phi_{ab}(n) = N\left\{\frac{abn}{N}\right\} = N\left\{\frac{a}{N} \cdot N\left\{\frac{bn}{N}\right\}\right\} = \phi_a (\phi_b(n)).
$$
Note then that there are signs $\e_1,\e_2,\e_3 \in \{-1,+1\}$ such that 
$$
\e_1 \lambda(abn) = \lambda(\phi_{ab}(n)) = \lambda(\phi_a (\phi_b(n))) = \e_2 \lambda(a\phi_b(n)) = \e_2 \lambda(a)\lambda(\phi_b(n)) = \e_2\e_3 \lambda(a)\lambda(bn),
$$
wherein 
\begin{itemize}
\item $\e_1 = \Lambda_{ab}(n) = -1$ iff $n \in \mc{E}_{ab}(N)$, 
\item $\e_2 = \Lambda_a(\phi_b(n)) = -1$ iff $\phi_b(n) \in \mc{E}_a(N)$ and 
\item $\e_3 = \Lambda_b(n) = -1$ iff $n \in \mc{E}_b(N)$. 
\end{itemize} 
Therefore, $\e_1 = -1$ if and only if precisely one of $\e_2$ and $\e_3$ is $-1$, so that
$$
\mc{E}_{ab}(N) = (\mc{E}_b(N) \cup \phi_b^{-1} (\mc{E}_a(N))) \bk (\mc{E}_b(N) \cap \phi_b^{-1}(\mc{E}_a(N))) = \mc{E}_b(N) \triangle \phi_b^{-1}(\mc{E}_a(N)).
$$
Swapping the roles of $a$ and $b$, we have symmetrically that $\phi_{ab}(n) = \phi_b (\phi_a(n))$, leading to the second decomposition
$$
\mc{E}_{ab}(N) = \mc{E}_a(N) \triangle \phi_a^{-1}(\mc{E}_b(N)).
$$
Thus, for any $a,b \geq 2$ coprime to $N$ we obtain the equality of sets
$$
\mc{E}_a(N) \triangle \phi_a^{-1}(\mc{E}_b(N)) = \mc{E}_b(N) \triangle \phi_b^{-1}(\mc{E}_a(N)).
$$
From this identity, we deduce
\begin{align*}
\mc{E}_b(N) \bk \phi_a^{-1}(\mc{E}_b(N)) &\subseteq [\left(\mc{E}_b(N) \bk \phi_b^{-1}(\mc{E}_a(N))\right) \bk \phi_a^{-1}(\mc{E}_b(N))] \cup \phi_b^{-1}(\mc{E}_a(N)) \\
&\subseteq [(\mc{E}_b(N) \triangle \phi_b^{-1}(\mc{E}_a(N)) \bk \phi_a^{-1}(\mc{E}_b(N))] \cup \phi_b^{-1}(\mc{E}_a(N)) \\
&= [(\mc{E}_a(N) \triangle \phi_a^{-1}(\mc{E}_b(N)) \bk \phi_a^{-1}(\mc{E}_b(N))] \cup \phi_b^{-1}(\mc{E}_a(N)) \\
&\subseteq \mc{E}_a(N) \cup \phi_b^{-1}(\mc{E}_a(N)).
\end{align*}
By a similar argument, we obtain also that
$$
\phi_a^{-1}(\mc{E}_b(N)) \bk \mc{E}_b(N) \subseteq \phi_b^{-1}(\mc{E}_a(N)) \cup \mc{E}_a(N).
$$
Since by Lemma \ref{lem:phidInj} $\phi_b$ is a bijection, we obtain that
$$
|\mc{E}_b(N) \triangle \phi_a^{-1}(\mc{E}_b(N))| \leq |\mc{E}_a(N) \cup \phi_b^{-1}(\mc{E}_a(N))| \leq |\mc{E}_a(N)| + |\phi_b^{-1}(\mc{E}_a(N))| = 2|\mc{E}_a(N)|,
$$
as claimed.
\end{proof}
\subsection{First observations about $g(d)$}
Using Lemma \ref{lem:phidInj}, we may establish two basic cases of our bounds on $g(d)$. The arguments will help motivate our treatment of more general values of $d$ in Section \ref{subsec:Pierce}.
\begin{lem} \label{lem:g23}
Suppose $(N,6) = 1$. Then $g(2) \leq 2$ and $g(3) \leq 6$.  
\end{lem}
\begin{proof}
Consider first the case $d = 2$. By \eqref{eq:Edinit}, $\Lambda_2(n) = +1$ unless $N/2 < n < N$. In this case, we have $N-n \in (0,N/2)$, whence $\Lambda_2(N-n) = + 1$. Note also that $\{m-\alpha\} = 1-\{\alpha\}$ for $0 < \alpha < m$ and $m \in \mb{N}$. Recalling that $m \in \mc{E}(N)$ if and only if $\lambda(m) = \eta \lambda(N-m)$, we have the relation
\begin{align*}
+1 = \Lambda_2(N-n) &= \lambda(2(N-n)) \lambda(N\{2(N-n)/N\}) = \lambda(2)\lambda(N-n) \lambda (N - N\{2n/N\}) \\
&= \lambda(2)(-\eta \lambda(n))(-\eta \lambda(N\{2n/N\})) =  \Lambda_2(n),
\end{align*}	
as long as neither $n$ nor $N\{2n/N\} = \phi_2(n)$ belongs to $\mc{E}(N)$. By Lemma \ref{lem:phidInj}, there is at most one $n$ that maps to each $\phi_2(n) \in \mc{E}(N)$. Thus, the number of exceptions to $\Lambda_2(n) = +1$ satisfies $|\mc{E}_2(N)| \leq 2|\mc{E}(N)|$  whence $g(2) \leq 2$ as claimed. \\
The case $d = 3$ is slightly more complicated. We note once again that, using \eqref{eq:Edinit}, $\Lambda_3(n) = +1$ unless $n \in (N/3,N)$. If $n \in (N/2,N)$ then as in the case $d = 2$ we have
$$
\Lambda_3(N-n) = \lambda(3) \lambda(N-n) \lambda(N-\phi_3(n)) = \Lambda_3(n),
$$
unless $n$ or $\phi_3(n) \in \mc{E}(N)$. Since $\phi_3$ is injective by Lemma \ref{lem:phidInj}, with $\leq 2|\mc{E}(N)|$ exceptions $n$ we can thus associate each $N/2 < n < N$ with a unique $n' = N-n \in (0, N/2)$ such that $\Lambda_3(n') = \Lambda_3(n)$. This can only yield an element of $\mc{E}_3(N)$ if $n' \in (N/3, N/2)$. We thus find that
\begin{equation}\label{eq:redtocentre}
|\mc{E}_3(N)| = |\mc{E}_3(N) \cap (N/3,N/2)| + |\mc{E}_3(N) \cap (N/2,N)| \leq 2|\mc{E}_3(N) \cap (N/3,N/2)| + 2|\mc{E}(N)|.
\end{equation}
It remains to consider when $n \in (N/3,N/2)$. In this case we see that $N-2n \in (0,N/3)$, so that $\Lambda_3(N-2n) = +1$. Observe also that unless $2n$ or $3(N-2n) \in \mc{E}(N)$, we get that $\lambda(N-2n) = -\eta\lambda(2n)$ and
$$
\lambda(3(N-2n)) = -\eta\lambda(N-3(N-2n)) = -\eta\lambda(6n-2N) = -\eta\lambda(2)\lambda(3n-N) = -\eta\lambda(2)\lambda(\phi_3(n)).
$$
Since the maps $n\mapsto 2n$ and $n \mapsto 3(N-2n)$ are injective, it follows that for all but $\leq 2|\mc{E}(N)|$ new exceptions from $(N/3,N/2)$,
$$
\Lambda_3(n) = \lambda(3n) \lambda(\phi_3(n)) = -\eta\lambda(2)\lambda(3n)\lambda(3(N-2n)) = -\eta\lambda(2n)\lambda(N-2n) = +1.
$$
Combined with \eqref{eq:redtocentre}, this gives
$$
|\mc{E}_3(N)| \leq 2|\mc{E}(N)| + 4|\mc{E}(N)| \leq 6 |\mc{E}(N)|,
$$
and $g(3) \leq 6$, as claimed.
\end{proof}
The following simple lemma reveals that $g$ satisfies a weak subadditivity property, which will be crucial to the inductive proof of Proposition \ref{prop:recur}. 
\begin{lem}\label{lem:composite}
For any $d\geq 2$ coprime to $N$ we have
$$
g(d) \leq \Omega(d) \sum_{p^k||d} k g(p).
$$
\end{lem}
\begin{proof}
We will show the following more general statement: if $k \geq 1$ and $1 < m_1,\ldots,m_k < N$ are all coprime to $N$ then
\begin{equation}\label{eq:addpropG}
g(m_1\cdots m_k) \leq k (g(m_1) + \cdots + g(m_k)).
\end{equation}
The claim follows from this on factoring $d = p_1 \cdots p_k$, where $p_1 \leq \cdots \leq p_k$ are (not necessarily distinct) primes, and collecting the common values of $g(p_j)$. \\
Observe that for each $a \in (\mb{Z}/N\mb{Z})^\times$ we get the telescoping sum
\begin{align*}
&S_{\lambda}(m_1\cdots m_k a/N) - \lambda(m_1\cdots m_k) S_{\lambda}(a/N) \\
&= \left(S_{\lambda}(m_1\cdots m_k a/N) - \lambda(m_k)S_{\lambda}(am_1\cdots m_{k-1}/N)\right) + \lambda(m_k) \left(S_{\lambda}(am_1\cdots m_{k-1}/N) - \lambda(m_{k-1})S_{\lambda}(am_1\cdots m_{k-2}/N)\right) \\
&+\cdots + \lambda(m_2\cdots m_k) \left(S_{\lambda}(am_1/N) - \lambda(m_1)S_{\lambda}(a/N)\right) \\
&= \sum_{1 \leq j \leq k} \lambda(m_{j+1}\cdots m_k) \left(S_{\lambda}(am_1\cdots m_j/N) - \lambda(m_j) S_{\lambda}(am_1\cdots m_{j-1}/N)\right),
\end{align*}
interpreting the products $m_{j+1}\cdots m_k$ and $m_1\cdots m_{j-1}$ as $1$ when $j = k$ and $j = 1$, respectively. By the Cauchy-Schwarz inequality, 
\begin{align*}
&\frac{1}{N}\sum_{a \pmod{N}} \left|S_{\lambda}\left(\frac{m_1\cdots m_k a}{N}\right) - \lambda(m_1\cdots m_k)S_{\lambda}\left(\frac{a}{N}\right)\right|^2 \\
&= \frac{1}{N}\sum_{a \pmod{N}} \left|\sum_{1 \leq j \leq k} \lambda(m_{j+1}\cdots m_k) \left(S_{\lambda}\left(\frac{am_1\cdots m_j}{N}\right) - \lambda(m_j)S_{\lambda}\left(\frac{am_1\cdots m_{j-1}}{N}\right)\right)\right|^2 \\
&\leq k \sum_{1 \leq j \leq k} \frac{1}{N} \sum_{b \pmod{N}} \left|S_{\lambda}\left(\frac{bm_j}{N}\right) - \lambda(m_j)S_{\lambda}\left(\frac{b}{N}\right)\right|^2,
\end{align*}
where in the last line we made the change of variables $b \equiv am_1\cdots m_{j-1} \pmod{N}$ for each $1 \leq j \leq k$. In light of Lemma \ref{lem:reltoEd}, we thus find that 
$$
4g(m_1\cdots m_k) |\mc{E}(N)| = 4|\mc{E}_{m_1\cdots m_k}(N)| \leq k \sum_{1 \leq j \leq k} \left(4 |\mc{E}_{m_j}(N)|\right) = 4|\mc{E}(N)| \left(k \sum_{1 \leq j \leq k} g(m_j)\right).
$$
This implies the claim.
\end{proof}
\subsection{The sets $\mc{E}_p(N)$ and Pierce expansions} \label{subsec:Pierce}
In light of the previous two lemmas, we need to obtain bounds for $|\mc{E}_p(N)|$ with prime $p > 3$, and to do so we revisit the arguments of \cite[Sec. 3]{Man} towards relating the dilated exponential sum values $S_{\lambda}(ap/N)$ to $S_{\lambda}(a/N)$, for prime $N$.  Our exposition will look significantly different here compared to what appears in \cite{Man}.\\
Let $p$ be prime, and let $1 \leq n,m < N$ with $np \equiv m \pmod{N}$. Then we may write 
$$
np = m+jN, \text{ where } j = \lfloor np/N\rfloor \text{ and } m = N\{np/N\} = \phi_p(n).
$$
Given $1 \leq n < N$, define 
$$
\theta(n) := n\{N/n\} = N - n\lfloor N/n\rfloor.
$$ 
This is a self-map on $\mb{Z} \cap (0,N)$, which we will iterate in the sequel. In \cite[Sec. 3.1]{Man}, for each $r \geq 1$ we constructed a map $\psi_r$ with the property that if $n \in (\tfrac{N}{r+1},\tfrac{N}{r})$ then $\psi_r$ mapped the pair $(m,j) = (\phi_p(n),\lfloor np/N\rfloor)$ to an image pair $(m',j')$, satisfying $0 \leq j' < j$ and
$$
m'+j'N = pN - r(m+jN) = p(N-rn) = p\theta(n) \in \left(0, \frac{pN}{r+1}\right)
$$ 
(here we are using the fact that $N$ is prime, as otherwise $\theta(n) = 0$ when $n = N/r$, say). In particular, $(m',j') = \left(\phi_p(\theta(n)), \left\lfloor p\theta(n)/N\right\rfloor\right)$, and 
$$
\theta(n) \in \left(0,\frac{N}{r+1}\right) \text{ whenever } n \in \left(\frac{N}{r+1}, \frac{N}{r}\right).
$$
Moreover, provided $rn, \phi_p(rn) \notin \mc{E}(N)$ so that
$$
\lambda(rn) = -\eta\lambda(N-rn) \text{ and } \lambda(\phi_p(\theta(n))) = -\eta\lambda(N-\phi_p(\theta(n)))
$$ 
then \cite[Lem. 3.1(d)]{Man} gives the explicit transformation
\begin{align*}
\Lambda_p(n) &= \lambda(pn)\lambda(\phi_p(n)) = \lambda(m+jN)\lambda(m) = \lambda(rm) \lambda(N\{rm/N\}) \lambda(m'+j'N) \lambda(m') \\
&= \lambda(r\phi_p(n)) \lambda(\phi_r(\phi_p(n))) \lambda(p\theta(n)) \lambda(\phi_p(\theta(n))) \\
&= \Lambda_r(\phi_p(n)) \Lambda_p(\theta(n)).
\end{align*}
(This is a generalisation of the arguments in Lemma \ref{lem:g23}.) Now, since the maps $n \mapsto rn$ and $n \mapsto \phi_p(n)$ are injective by Lemma \ref{lem:phidInj}, the number of exceptions that this process incurs is $\leq 2 |\mc{E}(N)|$.
%
Finally, we have that $0 \leq j' < j$ (equivalently, $0 \leq \left\lfloor p\theta(n)/N\right\rfloor < \left\lfloor pn/N\right\rfloor$), so the process can only be iterated at most $p$ times before $j' = 0$ (equivalently, $\theta^{(k)}(n) < N/p$ for some $k\geq 1$), with $\leq 2|\mc{E}(N)|$ exceptions in each instance. Once $j' = 0$ we get immediately that $\Lambda_p(\theta^{(k)}(n)) = \lambda(m') \lambda(m'+j'N) = +1$ by \eqref{eq:Edinit}. Thus, with the exception of at most $\leq 2p|\mc{E}(N)|$ elements $n$, we get that
$$
\Lambda_p(n) = \prod_{0 \leq j \leq k_p-1} \Lambda_{r_{j+1}}(\phi_p(\theta^{(j)}(n)))
$$
where for each $j \geq 1$, 
$$
r_{j+1} := \lfloor N/\theta^{(j)}(n)\rfloor,
$$ 
and $k_p=k_p(n)$ is the minimum index $k$ for which $r_{k+1} \geq p$ (equivalently, $\theta^{(k_p)}(n) < N/p$). We note that $\theta^{(j+1)}(n) = \theta^{(j)}(n) \{N/\theta^{(j)}(n)\} < \theta^{(j)}(n)$, and, writing $n_{\ell} := \theta^{(\ell)}(n) = n_{\ell-1} \{N/n_{\ell-1}\}$ for $0 \leq \ell < k_p$,
$$
r_{j+1} = \left\lfloor\frac{N}{n_{j-1}} \left\{\frac{N}{n_{j-1}}\right\}^{-1}\right\rfloor = 1 + \left\lfloor \left\lfloor \frac{N}{n_{j-1}} \right\rfloor \left\{\frac{N}{n_{j-1}}\right\}^{-1}\right\rfloor \geq 1+r_j.
$$
Therefore, $k_p(n)$ is well-defined with $1 \leq k_p(n) < p$ for all $1 \leq n < N$. As in \cite{Man}, we refer to the tuple $\sg_p(n) := (r_1,\ldots,r_k)$ as the \emph{$p$-signature} of $n$.
We thus summarise the above observations as follows.
\begin{lem} \label{lem:recurrence}
Let $2 \leq p < N$ be prime. Then, with the exception of $\leq 2p|\mc{E}(N)|$ integers $n < N$, whenever $n$ has $p$-signature $\sg_p(n) = (r_1,\ldots,r_{k_p})$, we have
\begin{equation} \label{eq:LambdaRs}
\Lambda_p(n) = \prod_{1 \leq j \leq k_p} \Lambda_{r_{j}}(\phi_p(\theta^{(j-1)}(n))).
\end{equation}
%
\end{lem}
\noindent Lemma \ref{lem:recurrence} indicates that we may obtain recursive information about $\mc{E}_p(N)$ in terms of the sets $\mc{E}_r(N)$ with $1 \leq r < p$ arising in the algorithm described above.\\
Interestingly, the elements of the $p$-signature are associated with the \emph{Pierce expansion} (also called the \emph{Ostrogradsky expansion} by some authors) of the rational number $n/N$ (see e.g. \cite{Sha} for a nice exposition), which we now describe. \\
Given $\alpha \in (0,1]$, the Pierce expansion of $\alpha$ is the expression
$$
\alpha = \frac{1}{r_1} - \frac{1}{r_1r_2} + \frac{1}{r_1r_2r_3} - \cdots,
$$
given by a (possibly infinite) strictly increasing sequence $(r_j)_{j\geq 1}$ of positive integers $r_j = r_j(\alpha)$. It is uniquely determined unless the expansion terminates, which occurs if and only if $\alpha \in \mb{Q}$. In this case, we can write
$$
\alpha = \frac{1}{r_1} - \cdots + \frac{(-1)^N}{r_1\cdots r_{N+1}},
$$
which is also a uniquely determined representation once we assert that\footnote{This restriction avoids the possible ambiguity arising from the substitution
$\tfrac{1}{r-1} = \tfrac{1}{r} - \tfrac{1}{r(r-1)}$. Though we shall work with Pierce expansions of rationals here, we will not employ the full Pierce expansion of any of these (since as $p < N$ here, each $r_j < p < N$ so that $\theta^{(j+1)}(n) > 1$; thus, we can never produce $n/N$ using a Pierce expansion that uses only $r_j < p$), and therefore this point of ambiguity is irrelevant to us.}
 $r_N < r_{N+1} - 1$.
The following lemma illustrates the relationship between the signature of $n$ and the Pierce expansion of $n/N$.
\begin{lem} \label{lem:preimages}
Let $n < N$, and given $k \geq 1$ set $r_j := r_j(n/N)$ for all $1 \leq j \leq k$. Then 
$$
n = N\sum_{1 \leq j \leq k} \frac{(-1)^{j-1}}{r_1\cdots r_j} + \theta^{(k)}(n)\frac{(-1)^{k}}{r_1\cdots r_{k}}.
$$
\end{lem}
\begin{proof}
As above, write $n_j$ to denote $\theta^{(j)}(n)$, and set $\nu_j := n_j/N$ for each $0 \leq j \leq k-1$. Then
$$
n_{j+1} = N-r_{j+1}n_j \Leftrightarrow \nu_{j+1} = 1-r_{j+1}\nu_j.
$$
Rearranging this and arguing by induction, we see that 
\begin{align*}
n &= N\nu_0 = \frac{N}{r_1}\left(1-\nu_1\right) = \frac{N}{r_1}\left(1-\frac{1}{r_2}\left(1-\nu_2\right)\right) \\
&= N\left(\frac{1}{r_1} - \frac{1}{r_1r_2} + \frac{\nu_2}{r_1r_2}\right) = \cdots \\
&= N\sum_{1 \leq j \leq k} \frac{(-1)^{j-1}}{r_1\cdots r_j} +  \frac{(-1)^{k} N\nu_k}{r_1\cdots r_k}.
\end{align*}
Since $N\nu_k = \theta^{(k)}(n)$,
the claim follows. 
\end{proof}
Combining Lemmas \ref{lem:composite}, \ref{lem:recurrence} and \ref{lem:preimages}, we may now prove our recursive bound, Proposition \ref{prop:recur}.
\begin{proof}[Proof of Proposition \ref{prop:recur}]
We will prove the first claim by induction. When $r = 2$ we already have $g(r) \leq 2 \leq 16$ by Lemma \ref{lem:g23}. \\
Next, let $R \geq 3$ and assume that $g(r) \leq 2^{r^2}$ for all $r < R$. In particular, it holds for all such \emph{prime} $r$. If $R$ is composite then as $P^+(R) \leq R/2$, Lemma \ref{lem:composite} yields
\begin{align} \label{eq:gdcomp}
g(R) \leq \Omega(R) \sum_{q^k||R} kg(q) \leq \Omega(R)^2 2^{P^+(R)^2} \leq (2\log R)^2 2^{(R/2)^2} \leq 2^{R^2}
\end{align}
It remains to consider when $R = p$ is prime. Representing the indicator function for $\mc{E}_p(N)$ as $1_{\mc{E}_p(N)}(n) = (1-\Lambda_p(n))/2$ and applying Lemma \ref{lem:recurrence}, we have
\begin{align*}
|\mc{E}_p(N)| &= \frac{1}{2}\sum_{n < N} (1-\Lambda_p(n)) = \frac{1}{2}\sum_{\ss{1 \leq r_1 < \cdots < r_k < p \\ k \geq 1}} \sum_{\ss{n < N \\ \sg_p(n) = \mbf{r}}} (1-\Lambda_p(n)) \\
&\leq 2p|\mc{E}(N)| + \frac{1}{2}\sum_{\ss{1 \leq r_1 < \cdots < r_k < p \\ k \geq 1}}\sum_{\ss{n < N \\ \sg_p(n) = \mbf{r}}} \left(1-\Lambda_{r_1}(\phi_p(n)) \cdots \Lambda_{r_k}(\phi_p(\theta^{(k-1)}(n)))\right).
\end{align*}
Using the simple inequality $0 \leq 1-x_1\cdots x_k \leq \sum_{1 \leq j \leq k} (1-x_j)$ for $x_j \in \{-1,+1\}$, we get from the previous line that 
\begin{align*}
|\mc{E}_p(N)| &\leq 2p|\mc{E}(N)| + \frac{1}{2}\sum_{\ss{1 \leq r_1 < \cdots < r_k < p \\ k \geq 1}}\sum_{\ss{n < N \\ \sg_p(n) = \mbf{r}}} \sum_{1 \leq j \leq k} \left(1-\Lambda_{r_j}(\phi_p(\theta^{(j-1)}(n)))\right) \\
&= 2p|\mc{E}(N)| + \sum_{\ss{1 \leq r_1 < \cdots < r_k < p \\ k \geq 1}}\sum_{1 \leq j \leq k} |\{n < N : \sg_p(n) = \mbf{r}, \, \theta^{(j-1)}(n) \in \phi_p^{-1}(\mc{E}_{r_j}(N))\}|.
\end{align*}
Grouping together terms according to the value of $r_j(n/N)$, and noting that since the sequence $(r_j(n/N))_j$ is increasing, at most one $j$ can yield $r_j(n/N) = r$, we get
\begin{align*}
|\mc{E}_p(N)| \leq 2p|\mc{E}(N)| + \sum_{1 \leq r < p} |\{n  < N : \exists j \geq 1 \text{ with } r_j(n/N) = r, \theta^{(j-1)}(n) \in \phi_p^{-1}(\mc{E}_r(N))\}|.
\end{align*}
By definition we have $r_j(n/N) = r$ if and only if
$\theta^{(j-1)}(n) \in (N/(r+1),N/r)$, for prime $N$. We thus find that
\begin{align}\label{eq:nurIntro}
|\mc{E}_p(N)| &\leq 2p|\mc{E}(N)| + \sum_{1 \leq r < p} \sum_{\ss{N/(r+1) < m < N/r \\ m \in \phi_p^{-1}(\mc{E}_r(N))}}  |\{n  < N : \exists j \geq 1 \text{ with } \theta^{(j-1)}(n) = m \}| \nonumber\\ 
&=: 2p|\mc{E}(N)| + \sum_{1 \leq r < p} \sum_{\ss{N/(r+1) < m < N/r \\ m \in \phi_p^{-1}(\mc{E}_r(N))}} \nu_r(m).
\end{align}
By Lemma \ref{lem:preimages}, the number $\nu_r(m)$ of integers $1 \leq n < N$ such that $\theta^{(j-1)}(n) = m$ and $r_j(n/N) = r$ for some $j \geq 1$ is bounded above by the number of non-empty subsets $R=\{r_1,\ldots,r_{j-1}\} \subseteq \{1,\ldots,r-1\}$, with any such subset $R$ giving the \emph{potential} preimage
$$
n_R = N\sum_{1 \leq i \leq j-1} \frac{(-1)^{i-1}}{r_1\cdots r_i} + m\frac{(-1)^{j-1}}{r_1\cdots r_{j-1}}
$$
(note however that this is \emph{not} guaranteed to be an integer). Therefore, $\nu_r(m) \leq 2^{r-1}$ uniformly over all $N/(r+1) < m < N/r$, and hence
\begin{align} \label{eq:toRecur}
|\mc{E}_p(N)| 
&\leq 2p |\mc{E}(N)| +\sum_{1 \leq r < p} 2^{r-1}|\phi_p^{-1}(\mc{E}_r(N)) \cap (N/(r+1),N/r)|.
\end{align}
Applying Lemma \ref{lem:phidInj}, we (crudely) have
$$
|\phi_p^{-1}(\mc{E}_r(N)) \cap (N/(r+1),N/r)| \leq |\phi_p^{-1}(\mc{E}_r(N))| = |\mc{E}_r(N)| = g(r)|\mc{E}(N)|.
$$
Inserting this into \eqref{eq:toRecur} and dividing through by $|\mc{E}(N)|$, we thus obtain the relation
$$
g(p) \leq 2p + \sum_{1 \leq r < p} 2^{r-1} g(r).
$$
Applying our induction hypothesis $g(r) \leq 2^{r^2}$ for all $r < p$, we finally obtain
$$
g(p) \leq 2p + \sum_{1 \leq r < p} 2^{r^2+r-1} \leq 2p + 2^{(p-1)^2 + p-2} + (p-2) 2^{(p-2)^2 + p-3} < 2^{p^2},
$$
for any $p \geq 3$. The first claim therefore follows. \\
To prove the second claim, we combine Lemma \ref{lem:composite} with the bound $\Omega(r) \leq 2\log r$ to get that if $P^+(r) \leq q$ then
$$
g(r) \leq \Omega(r)^2 \max_{p |r} g(p) \leq 4(\log r)^2 2^{q^2},
$$
and the claim follows.
\end{proof}
\begin{rem} \label{rem:nurprob}
The recursive bound $g(r) \leq 2^{r^2}$ given here would benefit significantly from a more refined analysis of the map $\nu_r(m)$ defined in \eqref{eq:nurIntro}. For instance, we show in the appendix to this paper that
$$
\sum_{N/(r+1) < n < N/r} \nu_r(n) \ll \frac{N \log r}{r},
$$
which suggests that $\nu_r(m)$ is of size\footnote{We thank Tom Bloom for pointing out an error in this estimate in a previous version of this argument.} $r \log r$ on average, and thus a loss of $2^{r-1}$, while potentially unavoidable pointwise in $m$, can be significantly improved for typical $m$. \\
On a purely heuristic level, if we expect that $\nu_r(m) \approx r \log r$ on average on $\phi_p^{-1}(\mc{E}_r(N)) \cap (N/(r+1),N/r)$ as well, and that $|\phi_p^{-1}(\mc{E}_r(N)) \cap (N/(r+1),N/r)| \approx |\mc{E}_r(N)|/r(r+1)$ then the recursive bound may be improved (heuristically) to
$$
g(p)|\mc{E}(N)| \lessapprox 2p|\mc{E}(N)| +\sum_{1 \leq r < p} r \log r \cdot \frac{|\mc{E}_r(N)|}{r^2} = |\mc{E}(N)|\left(2p + \sum_{1 \leq r < p} \frac{g(r) \log r}{r}\right),
$$
which admits a recursive solution of the shape $g(r) = e^{c(\log r)^2}$, and expands significantly the realm of application of the lemma. \\
In the context of our main results, moreover, it would allow a far larger choice of $q$. For instance, in order to get a non-trivial lower bound for the number of signs $(+,+)$ or $(-,-)$, it suffices to get a contradiction to the condition $|\mc{E}(N)| \ll Ne^{-\sqrt{\log N}}$ (see Section \ref{subsec:prop14}). Assuming we get an upper bound of  the shape 
$$
|\mc{E}_b(N)| \ll (\log b)^2 e^{c(\log q)^2}|\mc{E}(N)| \text{ for all } P^+(b) \leq q
$$
now allowing $q = \exp(c'(\log N)^{1/4})$ expands the range of $b$ to any $1 \leq b<N$. It is clear that, then we would attain ranges $P$ such that
$$
\max_{\ss{\chi \pmod{N} \\ \chi \neq \chi_0}} \frac{1}{\pi(2P)-\pi(P)} \left|\sum_{p \sim P} \chi(p)\right| = o(1),
$$
under a far weaker zero-free region\footnote{For example, a Littlewood zero-free region of the shape $\text{Re}(s) > 1-\tfrac{c\log\log N}{\log N}$, with $c > 0$ sufficiently small, and $|\text{Im}(s)| \leq (\log N)^3$ should suffice.} assumption. \\
If we punted the objective of obtaining a ``pure'' recursive bound $|\mc{E}_d(N)| \leq f(d)|\mc{E}(N)|$, for some function $f$, and instead allowed for additional $o(N)$ terms then more options are available; however, entwining them with a sum over $1 \leq r < d$ makes it a delicate matter to determine which such $o(N)$ terms would admit an improved bound on $|\mc{E}_d(N)|$. \\
A second moment estimate for $\nu_r$ might allow further savings (via H\"{o}lder's inequality) from conditioning on $n$ belonging to the \emph{a priori} sparse sets $\mc{E}_r(N)$. However, the iterative nature of the definition of $\nu_r(m)$ makes obtaining such second moment bounds somewhat challenging.
We plan to return to this task on another occasion.
\end{rem}
\section{Proof of Proposition \ref{prop:approxDil}} \label{sec:discrep}
In the previous section we showed that if $|\mc{E}(N)| = o(N)$ and $q$ is chosen to be suitably slowly growing as a function of $N$, then $|\mc{E}_r(N)| = o(N)$ in a relatively wide range of $q$-friable integers $r$. Since $q = O(\sqrt{\log N})$, however, this estimate is of limited direct utility. \\
In order to upgrade that result to non-friable $r$ we will study the distribution of $\mc{E}_b(N)$ inside of $\mb{Z} \cap (0,N)$, bolstered crucially by Lemma \ref{lem:shiftEb}. That result shows that for (possibly non-friable) large $b$, $\mc{E}_b(N)$ is well-covered by the shifts $\phi_a^{-1}(\mc{E}_b(N))$ for any $a \geq 2$ for which $|\mc{E}_a(N)|$ is small. We will therefore use this observation by taking $a$ to be $q$-friable as above. \\
Our key result in this section is the following estimate for exponential sums supported on $\mc{E}_b(N)$, which we will use to obtain bounds on the discrepancy of $\{n/N : n \in \mc{E}_b(N)\} \subseteq [0,1]$ (in the sense of uniform distribution theory).
\begin{prop} \label{prop:unifExpsumEb}
Let $N$ be a large prime. Let $k \in \mb{Z} \bk \{0\}$ with $(k,N) = 1$. Let $q^{\log q} \leq T \leq e^{q^{1/5}}$. Then for any $1 \leq b < N$ 
$$
\left|\sum_{n \in \mc{E}_b(N)} e(kn/N)\right| \ll NT^{-\frac{1}{20}+o(1)} + (\log T)^2 2^{q^2}|\mc{E}(N)|.
$$
\end{prop}
\begin{rem}
Observe that the left-hand side is independent of $T$, which makes this estimate applicable in principle to the entire range of $1 \leq b < N$. 
\end{rem}
The proof of Proposition \ref{prop:unifExpsumEb} is built on the following simple consequence of Lemma \ref{lem:shiftEb}. Given $1 \leq y \leq x$ we recall the standard notation
$$
S(x,y) := \{n \leq x : P^+(n) \leq y\}, \quad \Psi(x,y) := |S(x,y)|.
$$
\begin{lem} \label{lem:approxinvar}
Let $2 \leq q \leq \sqrt{\log N}$ and let $q \leq T < N$. Then for any $b \geq 2$ 
and any $1$-bounded, $N$-periodic function $g: \mb{Z}/N\mb{Z} \ra \mb{C}$ we have
$$
\sum_{n \in \mc{E}_b(N)} g(n) = \sum_{n \in \mc{E}_b(N)} \left(\frac{1}{\Psi(T,q)} \sum_{a \in S(T,q)} g(an)\right) + O\left((\log T)^22^{q^2}|\mc{E}(N)|\right).
$$
\end{lem}
\begin{proof}
Let $a \in S(T,q)$. Proposition \ref{prop:recur} shows that $|\mc{E}_a(N)| \ll (\log T)^{2}2^{q^2}|\mc{E}(N)|$. As $g$ is periodic modulo $N$ and $\phi_a(m) \equiv am\pmod{N}$ it follows that $g(am) = g(\phi_a(m))$ for all $m \in \mb{Z} \cap (0,N)$. Combining Lemmas \ref{lem:phidInj} and \ref{lem:shiftEb}, together with the fact that $|g| \leq 1$, yields
\begin{align*}
\left|\sum_{n \in \mc{E}_b(N)} g(n) - \sum_{n \in \mc{E}_b(N)} g(an)\right| &= \left|\sum_{m \in \phi_a^{-1}(\mc{E}_b(N))} g(am) - \sum_{n \in \mc{E}_b(N)} g(an)\right|\\
&\leq |\mc{E}_b(N) \triangle \phi_a^{-1}(\mc{E}_b(N))| \leq 2|\mc{E}_a(N)| \\
&\ll (\log T)^{2}2^{q^2}|\mc{E}(N)|.
\end{align*}
Since this is uniform in $a \in S(T,q)$, the claim follows upon averaging over all $a \in S(T,q)$ and swapping orders of summation.
\end{proof}
We will apply Lemma \ref{lem:approxinvar} to the non-trivial additive characters $g(n) = e(kn/N)$, suggesting the importance of having to hand bounds on exponential sums over friable numbers. We will use two such results from the literature. 
The first is a special case of an estimate 
due to Harper \cite[Thm. 1]{Har}.
\begin{thm1}[Harper]
Fix $A > 3$ and let $(\log x)^A \leq y \leq x^{1/3}$. Let $\theta \in \mb{R}$, and let $c/m$ be a reduced rational with $|\theta - c/m| = \Delta/m$. Set $\tilde{m} := \max\{m,\Delta x\}$ and suppose that $\tilde{m} \leq \frac{1}{10}(x/y^3)^{1/2}$. Then as $x \ra \infty$,
$$
\left|\sum_{\ss{n \in S(x,y)}} e(n\theta)\right|\ll \frac{\Psi(x,y)}{\tilde{m}^{1/2(1-3/A)-o(1)}} (\log x)^{9/2}.
$$
\end{thm1}
\begin{proof}
When $\theta = c/m + \delta$, and $4m^2y^3 (1+|\delta|x)^2 \leq x$,  Theorem 1 of \cite{Har} yields the above estimate with the right-hand side replaced by
$$
\frac{\Psi(x,y)}{(m(1+|\delta| x))^{1/2}}(m(1+|\delta|x))^{\tfrac{3}{2}(1-\alpha(x,y))} u^{3/2}\log u \log x \sqrt{\log (1+|\delta|x) \log(my)},
$$
wherein $u = (\log x)/\log y$ and, as $y \geq (\log x)^A$,
$$
\alpha(x,y) = 1-\frac{\log(u\log(u+1))}{y} + o(1) \geq 1-\frac{1}{A} + o(1)
$$ 
(see \cite[Sec. 2.1]{Har} for this latter estimate). We therefore have $u \leq \log x$, $y,m \leq x$ and $|\delta| = \Delta/m$, so that $m(1+|\delta|x) \leq 2 \tilde{m}$. Thus, under the condition $\tilde{m} \leq \frac{1}{10}(x/y^3)^{1/2}$, the above bound holds, and
$$
u^{3/2} \log u \log x \sqrt{(\log(1+|\delta|x)\log(my)} \ll (\log x)^{3/2 + 2 + 1} = (\log x)^{9/2},
$$
whence follows the claim.
\end{proof}
When $y= x^{o(1)}$ and $\delta > 0$ is small enough, the range $x^{1/2-\delta} < m \leq x^{1/2}$ is not suitably covered by Harper's theorem. As Harper notes in \cite{Har} (see the comments following the statement of Theorem 1 there), for this range one can instead use an earlier result of Fouvry-Tenenbaum \cite{FoTe}.
\begin{thm1}[Th\'{e}or\`{e}me 13 of \cite{FoTe}]
Let $3 \leq y \leq \sqrt{x}$ and let $\theta = \frac{b}{m} + \delta$, for some $1 \leq b \leq m$ with $(b,m) = 1$. Then
$$
\left|\sum_{\ss{n \in S(x,y)}} e(n\theta)\right| \ll x(1+|\delta|x)(\log x)^3 \left(y^{1/2}x^{-1/4} + \frac{1}{\sqrt{m}} + \sqrt{\frac{my}{x}}\right).
$$
\end{thm1}
\begin{proof}[Proof of Proposition \ref{prop:unifExpsumEb}]
Let $k \in \mb{Z} \bk \{0\}$ with $(k,N) = 1$. Applying Lemma \ref{lem:approxinvar} with $g(n) = e(kn/N)$,
we obtain 
\begin{align}\label{eq:averagesmooth}
\sum_{n \in \mc{E}_b(N)} e(kn/N) 
&= \sum_{n \in \mc{E}_b(N)} \left(\frac{1}{\Psi(T,q)} \sum_{\ss{a \in S(T,q)}} e(kna/N)\right) + O((\log T)^22^{q^2}|\mc{E}(N)|).
\end{align}
Define $T_0 := q^{10}\sqrt{T}$ and $M_2 := T/T_0 \leq \sqrt{T}$. By the Dirichlet approximation theorem, for every $\alpha \in [0,1]$ there are $1 \leq m \leq T_0$ and $1 \leq c \leq m$ coprime to $m$ such that 
$$
\alpha \in \left(\frac{c}{m} - \frac{1}{mT_0}, \frac{c}{m} + \frac{1}{mT_0}\right).
$$ 
We thus decompose $\mc{E}_b(N)$ in the following way. Let $1 \leq M_1 < M_2$ be a parameter to be chosen later, and set
$$
\mc{M}_1 := \bigcup_{1 \leq m \leq M_1}\,  \bigcup_{\ss{c \pmod{m} \\ (c,m) = 1}} \left(\frac{c}{m} - \frac{1}{mT_0}, \frac{c}{m} + \frac{1}{mT_0}\right), \quad \mc{M}_2 := \bigcup_{M_1 < m \leq M_2} \, \bigcup_{\ss{c \pmod{m} \\ (c,m) = 1}} \left(\frac{c}{m} - \frac{1}{mT_0}, \frac{c}{m} + \frac{1}{mT_0}\right).
$$
Finally, define
$$
\mc{B} := \{n \in \mc{E}_b(N) : \{nk/N\} \in \mc{M}_1\}, \quad \mc{S} := \{n \in \mc{E}_b(N) : \{nk/N\} \in \mc{M}_2\}, \quad \mc{L} := \mc{E}_b(N) \bk (\mc{B} \cup \mc{S}).
$$ 
We observe that since $(k,N) = 1$, 
$$
|\mc{B}| \leq \sum_{1 \leq m \leq M_1} \sum_{c \pmod{m}} \left|\mb{Z} \cap \left(\frac{Nc}{m} - \frac{N}{mT_0}, \frac{Nc}{m} + \frac{N}{mT_0}\right)\right| \ll \sum_{1 \leq m \leq M_1} \sum_{c \pmod{m}} \frac{N}{mT_0} \ll \frac{NM_1}{T_0}.
$$
When $n \in \mc{L}$ we have 
$$
|\delta|T = \left|\left\{\frac{kn}{N}\right\} - \frac{c}{m}\right|T \leq \frac{T}{mT_0} < \frac{T}{M_2T_0} = 1.
$$
As $q \leq T^{o(1)}$, applying the Fouvry-Tenenbaum result yields
$$
\max_{n \in \mc{L}} \frac{1}{\Psi(T,q)} \left| \sum_{\ss{a \in S(T,q)}} e(kan/N)\right| \ll \frac{T(\log T)^{3}}{\Psi(T,q)}\left(q^{1/2}T^{-1/4} + \frac{1}{\sqrt{M_2}} + \left(\frac{qT_0}{T}\right)^{1/2}\right) \ll \frac{T^{3/4 + o(1)}}{\Psi(T,q)}.
$$
Let $A := 5$. Since $q \geq (\log T)^A$, we have 
$$
u = \frac{\log T}{\log q} \leq \frac{\log T}{A\log\log T}.
$$ 
Thus, provided $N$ (and thus $q$) is sufficiently large we apply a smooth numbers estimate due to Canfield, Erd\H{o}s and Pomerance \cite{CEP} to obtain
$$
\Psi(T,q) = T u^{-(1+o(1))u}  \gg T^{1-\tfrac{1}{A} + o(1)}.
$$
Thus, we obtain
$$
\max_{n \in \mc{L}} \frac{1}{\Psi(T,q)} \left| \sum_{\ss{a \in S(T,q)}} e(kan/N)\right| \ll T^{\tfrac{1}{A}-\tfrac{1}{4}+o(1)}.
$$
Finally, consider when $n \in \mc{S}$. Since
$$
\tilde{m} = \max\{m,\Delta T\} \leq \max\{M_2, T/T_0\} = T/T_0 \leq \sqrt{T}/q^{10} \leq \tfrac{1}{10} (T/q^3)^{1/2},
$$
we may apply Harper's theorem to obtain
$$
\max_{n \in \mc{S}} \frac{1}{\Psi(T,q)} \left|\sum_{\ss{a \in S(T,q)}} e(kan/N) \right| \ll \frac{(\log T)^{9/2}}{\tilde{m}^{\tfrac{1}{2}(1-3/A)-o(1)}} \leq \frac{(\log T)^{9/2}}{M_1^{\tfrac{1}{2}(1-3/A)-o(1)}}.
$$
Collecting these bounds together, we get from \eqref{eq:averagesmooth} that
\begin{align*}
&\left|\sum_{n \in \mc{E}_b(N)} e(kan/N)\right| \\
&\leq |\mc{B}| + |\mc{E}_b(N)| \left(\max_{n \in \mc{S}} \frac{1}{\Psi(T,q)} \left|\sum_{\ss{a \in S(T,q)}} e(kan/N)\right| + \max_{n \in \mc{L}} \frac{1}{\Psi(T,q)} \left|\sum_{\ss{a \in S(T,q)}} e(kan/N)\right|\right) + (\log T)^2 2^{q^2}|\mc{E}(N)| \\
&\ll \frac{NM_1}{T_0} + |\mc{E}_b(N)|\left(\frac{(\log T)^{9/2}}{M_1^{\tfrac{1}{2}(1-3/A)-o(1)}} + T^{\tfrac{1}{A}-\tfrac{1}{4}+o(1)}\right) + (\log T)^22^{q^2}|\mc{E}(N)|  \\
&\ll NT^{o(1)}\left(\frac{M_1}{T_0} + M_1^{-\tfrac{1}{2}(1-3/A)} + T^{\tfrac{1}{A}-\tfrac{1}{4}}\right) + (\log T)^2 2^{q^2}|\mc{E}(N)|.
\end{align*}
We choose $M_1 = T^{1/4}$ to obtain
$$
\left|\sum_{n \in \mc{E}_b(N)} e(kan/N)\right| \ll NT^{o(1)}\left(T^{\frac{1}{5} - \frac{1}{4}} + T^{-\tfrac{1}{8}(1-3/5)}\right)  + (\log T)^2 2^{q^2}|\mc{E}(N)| \ll NT^{-\tfrac{1}{20}+o(1)} + (\log T)^2 2^{q^2}|\mc{E}(N)|,
$$
as claimed.
\end{proof}
We are now in a position to prove Proposition \ref{prop:approxDil}.
\begin{proof}[Proof of Proposition \ref{prop:approxDil}]
Let $1 \leq K < N$ be a parameter to be chosen, and let $2 \leq r \leq T^{1/40}$. By the Erd\H{o}s-Tur\'{a}n inequality \cite[Cor. 1.1]{TenLec}, we have
$$
\sup_{I \subseteq [0,1]} \left||\{n \in \mc{E}_r(N) : n/N \in I\}| - |I| |\mc{E}_r(N)|\right| \ll \frac{|\mc{E}_r(N)|}{K} + \sum_{1 \leq k \leq K} \frac{1}{k} \left|\sum_{n \in \mc{E}_r(N)} e(nk/N)\right|.
$$
Since $N$ is prime, for every $1 \leq k \leq K$ we have $(k,N) = 1$. Thus, applying the previous lemma and specialising to $I = (0,1/r)$, we have
$$
\left||\{n \in \mc{E}_r(N) : n\in (0,N/r)\}| - \frac{1}{r}|\mc{E}_r(N)|\right| \ll \frac{|\mc{E}_r(N)|}{K} + (\log K) \left(N T^{-1/20+o(1)} + (\log T)^2 2^{q^2} |\mc{E}(N)|\right).
$$
But as remarked in \eqref{eq:Edinit}, we have $\mc{E}_r(N) \cap (0,N/r) = \emptyset$, so that upon rearranging we get
$$
|\mc{E}_r(N)| \ll rN\left(\frac{1}{K} + \frac{\log K}{T^{1/20+o(1)}}\right) + r(\log K)(\log T)^2 2^{q^2} |\mc{E}(N)|.
$$
We select $K = T^{1/20}$, so as $\log T \leq q$ and $r \leq T^{1/40}$ we conclude that
$$
|\mc{E}_r(N)| \ll NT^{-1/40+o(1)} + 2^{(1+o(1))q^2} |\mc{E}(N)|.
$$
Replacing $T$ by $T^{1/40}$ yields the statement of Proposition \ref{prop:approxDil}.
\end{proof}
\section{Proofs of the main results}
In this section we prove Theorems \ref{thm:Shus} and \ref{thm:LGcond}.  
\subsection{Proof of Theorem \ref{thm:LGcond}} \label{subsec:prop14}
Theorem \ref{thm:LGcond} is a consequence of Proposition \ref{prop:keyLowerBd}, which we may now obtain.
\begin{proof}[Proof of Proposition \ref{prop:keyLowerBd}]
Assume GRH. We begin with the proof of the first claim. Assume for the sake of contradiction that there is some $C > 0$ such that for infinitely many $N$,
$$
|\mc{L}_{\lambda}(N)| > N-Ne^{-C(\log \log N)^{6}}.
$$
Then there is an $\eta \in \{-1,+1\}$ such that
$$
|\mc{E}(N)| = |\{n < N : \lambda(n) = \eta\lambda(N-n)\}| = \frac{1}{2}\left(N-1 + \eta \mc{L}_{\lambda}(N)\right) < Ne^{-C(\log \log N)^{6}}.
$$
Now, let 
$$
q := \frac{C}{10}(\log \log N)^{6}, \quad P := e^{\tfrac{1}{40}q^{1/5}},
$$ 
so that if $N$ is sufficiently large then $P \geq (\log N)^{100}$. By Proposition \ref{prop:approxDil} whenever $P/2 < p \leq P$ we obtain 
$$
|\mc{E}_p(N)| \ll NP^{-1+o(1)} + 2^{(1+o(1)) q^2} |\mc{E}(N)|  = o(N).
$$
On the other hand, by Proposition \ref{prop:GRHavg}, we can find $p \in (P/2,P]$ such that $|\mc{E}_p(N)| \gg N$, which yields a contradiction. The first claim follows. \\
To prove the second claim, it suffices to observe that for each $(\eta_1,\eta_2) \in \{-1,+1\}^2$,
\begin{align*}
|\{n < N : (\lambda(n),\lambda(N-n)) = (\eta_1,\eta_2)\}| &= \frac{1}{4} \sum_{n < N} (1+\eta_1\lambda(n))(1+\eta_2\lambda(N-n)) \\
&= \frac{1}{4} \left(N-1 + \eta_1\eta_2 \mc{L}_{\lambda}(N) + (\eta_1+\eta_2) \sum_{n < N} \lambda(n)\right) \\
&\geq \frac{1}{4}\left(N-1 - |\mc{L}_{\lambda}(N)| - 2\left|\sum_{n < N} \lambda(n)\right|\right).
\end{align*}
Combining the first claim with the prime number theorem (with Vinogradov-Korobov error term), there are constants $C_1,C_2 > 0$ such that this last expression is
$$
\geq C_1 Ne^{-C(\log\log N)^{6}} - C_2 Ne^{-(\log N)^{3/5}(\log\log N)^{-1/5}} \gg Ne^{-C(\log \log N)^{6}}
$$
for large enough $N$, and the second claim follows.
\end{proof}
\begin{rem} \label{rem:GRHimp}
We show here that, instead of GRH the assumption that $L(s,\chi) \neq 0$ for all $\chi \neq \chi_0$ modulo $N$ in the rectangle
$$
\alpha < \text{Re}(s) \leq 1, \, |\text{Im}(s)| \leq (\log N)^3
$$
with $\alpha = 1-(\log N)^{-c}$ and $c \in (0,3/50)$, suffices to prove our two main theorems. \\
Assume that $|\mc{E}(N)| \leq Ne^{-(\log N)^{3/5}}$. We show that we may get a contradiction under these conditions, which suffices in the proof of Proposition \ref{prop:keyLowerBd} above. Following the proof of Proposition \ref{prop:GRHavg}, we have $|\mc{E}_p(N)| \gg N$ for some $P < p \leq 2P$, as long as $P = P(N) \ra \infty$ and
$$
\frac{(\log N)(\log\log N)}{P^{1-\alpha}} \ll \frac{1}{\log\log N},
$$
say. This last condition is fulfilled as long as we may select 
$$
\log P \geq 2(\log N)^c \log\log N, 
$$
taking $N$ sufficiently large. On the other hand, Proposition \ref{prop:approxDil} allows for a bound of the shape
$$
|\mc{E}_r(N)|  \ll 2^{(1+o(1))q^2}|\mc{E}(N)| + Nr^{-1+o(1)} = o(N),
$$
on taking $q = \tfrac{1}{100}(\log N)^{3/10}$, $R := e^{0.01 q^{1/5}}$ and $R/2 < r \leq R$. Then $\log R  \gg q^{1/5} = (\log N)^{3/50-o(1)}$. Thus, provided $c \in (0,3/50)$ the choice of $P = R/2$ is admissible in giving a contradiction, provided $N$ is sufficiently large.
\end{rem}
Having proven Proposition \ref{prop:keyLowerBd}, we may now obtain Theorem \ref{thm:LGcond}.
\begin{proof}[Proof of Theorem \ref{thm:LGcond} ]
For each $\mbf{\eta} := (\eta_1,\eta_2) \in \{-1,+1\}^2$ and prime $p \geq 3$ define
$$
T_{\mbf{\eta}}(p) := |\{n < p : (\lambda(n) , \lambda(p-n)) = (\eta_1,\eta_2)\}|.
$$
By Proposition \ref{prop:keyLowerBd}, 
$$
\min_{\mbf{\eta} \in \{-1,+1\}^2}T_{\mbf{\eta}}(p) \gg pe^{-C(\log \log p)^{6}},
$$
so in particular, there is $p_0 \geq 2$ such that for every prime $p \geq p_0$,
$$
\min_{\mbf{\eta} \in \{-1,+1\}^2} T_{\mbf{\eta}}(p) \geq 1.
$$
Assume now that $N$ is not $p_0$-friable. Then there is a prime $p > p_0$ such that $N = pN'$, for some $N' \in \mb{N}$. Set $\eta := \lambda(N')$, and let $\mbf{\eps} \in \{-1,+1\}^2$. By the above assumption, we may find $1 \leq b < p$ with 
\begin{align*}
(\lambda(b), \lambda(p-b)) = (\eta \eps_1, \eta \eps_2).
\end{align*}
Let $a := bN'$, so that $N-a = N'(p-b)$. Then 
$$
\lambda(a) = \lambda(b) \lambda(N') = (\eta \eps_1) \eta = \eps_1, \quad \lambda(N-a) = \lambda(N')\lambda(p-b) = \eta (\eta \eps_2) = \eps_2.
$$
The first claim thus follows. \\
For the second claim, it is clear that if $N_0 := \prod_{p \leq p_0} p$ and if $N > N_0$ is squarefree then $N$ has a prime factor $p > p_0$. Thus, as $N$ is not $p_0$-friable, (2) must hold for $N$, and the second claim follows from the first.
\end{proof}
\subsection{Application to Shusterman's problem}
By combining Theorem \ref{thm:LGcond} with the work of \cite{Man}, we may now prove Theorem \ref{thm:Shus}.
\begin{proof}[Proof of Theorem \ref{thm:Shus}]
Let $p_0$ be the constant from Theorem \ref{thm:LGcond}. We consider four cases:
\begin{enumerate}[(i)]
\item Suppose that $N$ is not $p_0$-friable. By Theorem \ref{thm:LGcond}, we immediately have a pair $1 \leq a,b < N$ with $a+b = N$ and $\lambda(a) = \lambda(b) = -1$, as claimed. 
\item Suppose $N$ is $p_0$-friable and divisible by $8$, and write $N = 8N'$. Then as $8 = 4 + 4 = 3 + 5$, we obtain the sign patterns $\lambda(3) = \lambda(5) = -1$ and $\lambda(4) = \lambda(4) = +1$. Thus, if $\lambda(N') = -1$ then we get the choice $a =b= 4N'$, yielding $\lambda(a) = \lambda(b) = -1$, while if $\lambda(N') = +1$ then the choice $a = 3N'$, $b = 5N'$ yields $\lambda(a) = \lambda(b) = -1$.
\item Suppose $N$ is $p_0$-friable, $8 \nmid N$ and $\lambda(N) = +1$. As $N$ is even, the choice $a = b = N/2$ immediately yields $\lambda(a) = \lambda(b) = -\lambda(N) = -1$. 
\item Finally, suppose $N$ is $p_0$-friable, $8 \nmid N$ and $\lambda(N) = -1$. We claim that there is a threshold $N_0$ such that if $N \geq N_0$ then a choice of $a$ and $b$ with $\lambda(a) = \lambda(b) = -1$ exists. \\
We observe first that any $p_0$-friable, even integer $N$ with $8 \nmid N$ may be written as $N = 2^k M^2 N'$, where $1 \leq k \leq 2$, $2 \nmid N'M^2$ and $M^2$ is the largest odd square factor dividing $N$. Then $N'$ is squarefree and $p_0$-friable, so e.g. by the prime number theorem there is a constant $C > 0$ such that
$$
2^kN' \leq 4\prod_{p \leq p_0} p \leq 4e^{Cp_0}.
$$
Now, set $N_0 := 4 \cdot 11^2 \cdot e^{Cp_0}$ and assume that $N \geq N_0$. Then, writing $N = 2^k M^2N'$ as before, $M \geq 11$. As $M$ is odd, applying \cite[Thm. 1]{Man} shows there is an $1 \leq n < M/2$ such that $\lambda(n) = \lambda(M-n)$. Setting $d := M-2n \geq 1$, we get that $n = (M-d)/2$, $M-n = (M+d)/2$, so that $\lambda(M-d) = \lambda(M+d)$, and thus $\lambda(M^2-d^2) = +1$. We thus find that the pair $(M^2-d^2,d^2)$ produces the sign pattern 
$$
(\lambda(M^2-d^2), \lambda(d^2)) = (+1,+1).
$$ 
Therefore, as $\lambda(2^{k}N') = \lambda(N) = -1$ the pair $(a,b) = (2^kN'(M^2-d^2),2^kN'd^2)$ yields
\begin{align*}
&(\lambda(a),\lambda(b)) = (\lambda(2^kN'(M^2-d^2)),\lambda(2^kN' d^2)) = (-1,-1), \\
&a + b = 2^{k}N'(M^2-d^2) + 2^{k}N'd^2 = 2^{k}N'M^2 = N.
\end{align*}
\end{enumerate}
Combining these four cases, the theorem follows.
\end{proof}

\section*{Appendix A: A first moment estimate for $\nu_r(m)$}
As discussed in Remark \ref{rem:nurprob}, our recursive bound $|\mc{E}_r(N)| \leq 2^{r^2} |\mc{E}(N)|$ would improve substantially provided we could better estimate sums involving the quantity
$$
\nu_r(m) := \left|\left\{n < N : \, \exists \, 1 \leq r_1 < \cdots < r_k < r \text{ with } n = N\sum_{1 \leq j < k} \frac{(-1)^{j-1}}{r_1\cdots r_j} + \frac{(-1)^k m}{r_1\cdots r_k}\right\}\right|.
$$
Improving on the pointwise bound $\nu_r(m) \leq 2^{r-1}$ used in the proof of Proposition \ref{prop:recur}, we may show the following average bound for $\nu_r(m)$ over the interval $(N/(r+1),N/r)$, for $r \geq 1$ not too large (for comparison, our current range of $r$ in the recursive bound is $r = O(\sqrt{\log N})$).
\begin{lem1}\label{eq:nurAvg}
Let $\delta >0$, let $r$ be a sufficiently large integer (in terms of $\delta$ only), and assume that $N \geq \exp(10(\log r)^{2+\delta})$. Then 
\begin{align*}
&\sum_{N/(r+1) < m < N/r} \nu_r(m) \ll \frac{N \log r}{r}.
\end{align*}
\end{lem1}
\begin{proof}
We begin with the following observation. Recall that $\nu_r(m)$ counts integers $n < N$ such that there is a tuple $1 \leq r_1 < \cdots < r_k < r$, $k \geq 1$ with
\begin{equation}\label{eq:nrep}
n = N\sum_{1 \leq j \leq k} \frac{(-1)^{j-1}}{r_1\cdots r_j} + \frac{(-1)^k m}{r_1\cdots r_k}.
\end{equation}
Suppose $N/(r+1) < m < N/r$ and suppose $n$ arises from a tuple $\mbf{r} = (r_1,\ldots,r_k)$ with $r_1\cdots r_k > 2N/r$. Let $2 \leq k_0 \leq k$ be the index such that 
$$
r_1\cdots r_{k_0-1} \leq \frac{2N}{r} < r_1\cdots r_{k_0} \leq 2N
$$
(that $k_0 \geq 2$ follows from $N > r^2 > rr_1$). Then $n$ is the closest integer to $N\sum_{1 \leq j \leq k_0} \frac{(-1)^{j-1}}{r_1\cdots r_j}$, and so is determined by the initial segment $(r_1,\ldots,r_{k_0})$ of the tuple $(r_1,\ldots,r_k)$, and arises for at most a single $m$ (which, then, may occur in an interval outside of $(N/(r+1),N/r)$). It follows that the number of $n$ determined in this way by tuples of $r_j$ satisfying $r_1\cdots r_k > 2N/r$, and $m \in (N/(r+1),N/r)$, is bounded by the number of tuples of $r_j'$ such that $r_1'\cdots r_k' \leq 2N$. \\
On the other hand, suppose $N/(r+1) < m < N/r$ and $n$ arises from a tuple $\mbf{r}$ with $r_1\cdots r_k \leq 2N/r$. In this case, from \eqref{eq:nrep} we see that
$$
m \equiv NP_k(r_1,\ldots,r_k) \pmod{r_1\cdots r_k},
$$
where we have set
$$
P_k(r_1,\ldots,r_k) = r_1\cdots r_k \left(\sum_{1 \leq j \leq k} \frac{(-1)^{j-1}}{r_1\cdots r_j}\right) \in \mb{Z}.
$$
We may therefore upper bound the sum over $\nu_r(m)$ as
\begin{align*}
&\sum_{N/(r+1) < m < N/r} \nu_r(m) \\
&\leq \sum_{\ss{1 \leq r_1 < \cdots < r_k < r \\ r_1\cdots r_k \leq 2N/r}} |\{N/(r+1) < m < N/r : m \equiv NP_k(r_1,\cdots, r_k) \pmod{r_1\cdots r_k}\}| \\
&+ |\{1 \leq r_1 < \cdots < r_k < r : r_1\cdots r_k \leq 2N\}| \\
&=: T_1 + T_2.
\end{align*}
In estimating $T_1$ we use that
$$
|\{N/(r+1) < m < N/r : m \equiv NP_k(r_1,\cdots, r_k) \pmod{r_1\cdots r_k}\}| \leq \left(1+ \frac{N}{r_1 \cdots r_k r(r+1)}\right),
$$
so that
$$
T_1 \leq N\sum_{\ss{1 \leq r_1 < \cdots < r_k < r \\ r_1\cdots r_k \leq 2N/r}} \frac{1}{r_1\cdots r_kr(r+1)} + |\{1 \leq r_1 < \cdots < r_k < r : r_1\cdots r_k \leq 2N\}| \leq T_1' + T_2,
$$
writing $T_1'$ to denote the sum in the middle expression. To estimate $T_1'$ we use Rankin's trick in the form $1_{r_k < r} \leq \frac{r}{r_k+1}$, 
to get
\begin{align*}
T_1' &\leq \frac{N}{r+1} \sum_{1 \leq r_1 < \cdots < r_k < r} \frac{1}{r_1\cdots r_{k-1} r_k(r_k+1)} =: \frac{N}{r+1}S(r).
\end{align*}
We may decompose $S(r)$ in terms of $S(r-1)$ as
\begin{align*}
S(r) &= S(r-1) + \sum_{\ss{1 \leq r_1 < \cdots < r_{k-1} < r-1 \\ r_k = r-1}} \frac{1}{r_1\cdots r_{k-1}r_k(r_k+1)} = S(r-1) + \frac{1}{r(r-1)} \sum_{1 \leq r_1 < \cdots < r_{k-1} < r-1} \frac{1}{r_1\cdots r_{k-1}} \\
&= S(r-1) + \frac{1}{r(r-1)}\prod_{1 \leq j < r-1} \left(1+\frac{1}{j}\right) = S(r-1) + \frac{1}{r}.
\end{align*}
Iterating this over $r$, we obtain $S(r) \ll \log r$, and hence
$$
T_1' \ll \frac{N \log r}{r}.
$$
We next consider $T_2$. We split this as $T_2 = T_2' + T_2''$, where 
\begin{align*}
T_2' &:= |\{1 \leq r_1 < \cdots < r_k < r : \, r_1\cdots r_k \leq 2N, \, 1 \leq k \leq (\log r)^{1+\delta}\}|, \\
T_2'' &:= |\{1 \leq r_1 < \cdots < r_k < r : \, r_1\cdots r_k \leq 2N, \, k > (\log r)^{1+\delta}\}|.
\end{align*}
We apply Rankin's trick to $T_2'$ to get
\begin{align*}
T_2' \leq \sum_{\ss{1 \leq r_1 < \cdots < r_k < r \\ 1 \leq k \leq (\log r)^{1+\delta}}} \left(\frac{2N}{r_1\cdots r_k}\right)^{1-1/\log r} \leq (2N)^{1-1/\log r} \sum_{\ss{1 \leq r_1 < \cdots < r_k < r \\ 1 \leq k \leq (\log r)^{1+\delta}}} \frac{e^k}{r_1\cdots r_k}.
\end{align*}
Repeating the argument of $T_1'$ and invoking the assumed bound $N^{1/(2\log r)} > r(\log r)e^{(\log r)^{1+\delta}}$, we see that
$$
T_2' \leq N^{1-1/\log r} e^{(\log r)^{1+\delta}} r \sum_{1 \leq r_1 < \cdots < r_k < r} \frac{1}{r_1\cdots r_k (r_k+1)} < N^{1-1/\log r} r (\log r) e^{(\log r)^{1+\delta}} \leq N^{1-1/(2\log r)}.
$$
Finally we consider $T_2''$. For this, we first observe (following \cite{Sha}) that if $K := \lfloor (\log r)^{1+\delta}\rfloor$ then
\begin{align*}
T_2'' \leq 2N \sum_{\ss{1 \leq r_1 < \cdots < r_k < r \\ k > (\log r)^{1+\delta}}} \frac{1}{r_1\cdots r_k} \leq 2N \sum_{k \geq K} \sum_{1 \leq r_1 < \cdots < r_k < r} \frac{1}{r_1\cdots r_k} = \frac{2N}{(r-1)!} \sum_{k \geq K} \left[{r \atop k}\right], 
\end{align*}
where $\left[{r \atop k}\right]$ is an unsigned Stirling number of the first kind. By explicit estimates for Stirling numbers (see e.g. \cite{Adell}), when $k > (\log r)^{1+\delta}$ we have
\begin{equation}\label{eq:Stirupp}
\frac{\left[{r \atop k}\right]}{(r-1)!} \leq \frac{(\tau + r-1) \cdots (\tau+1)}{\tau^k(r-1)!},
\end{equation}
where $\tau = \tau(r,k) > 0$ is the unique solution to the equation
\begin{equation}\label{eq:tauDef}
k = \sum_{1 \leq j \leq r-1} \frac{\tau}{\tau+j}.
\end{equation}
We observe that, crudely,
$$
k \leq \tau \sum_{1 \leq j \leq r-1} \frac{1}{j} \leq \tau(\log r + 1),
$$
so that $\tau \geq \tfrac{1}{2}(\log r)^{\delta} \geq 20$, provided that $r$ is sufficiently large. We use this shortly. \\
Using \eqref{eq:Stirupp} and \eqref{eq:tauDef}, we have
\begin{align*}
\frac{\left[{r \atop k}\right]}{(r-1)!} &\leq \exp\left(\sum_{1 \leq j \leq r-1} \left(\log(\tau + j) - \log j\right) - k \log \tau\right) \\
&= \exp\left(\sum_{1 \leq j \leq r-1} \left(\log(1+\tau/j) - \frac{\tau\log \tau}{\tau+j}\right)\right) \\
&= \exp\left(\sum_{1 \leq j \leq r-1} \frac{\tau}{\tau+j}\left(\left(1+\frac{j}{\tau}\right)\log\left(1+\frac{\tau}{j}\right) - \log \tau\right)\right).
\end{align*}
Next, observe that $\phi_{\tau}(u) := (1+u/\tau)\log(1+\tau/u)$ is non-increasing for all $u >0$, since
$$
\phi_{\tau}'(u) = \frac{1}{\tau} \log(1+\tau/u) - \frac{1}{u} \leq \frac{1}{\tau} \cdot \frac{\tau}{u} - \frac{1}{u} \leq 0.
$$
It follows that as $\tau \geq 20$, for every $j \geq 2$ we get
$$
\phi_{\tau}(j) - \log \tau \leq \phi_{\tau}(2) - \log \tau  = \left(1+\frac{2}{\tau}\right)\left(\log(2+\tau) - \log 2\right) - \log \tau \leq -\log 2 + \frac{2}{\tau} \left(1+\log(2+\tau)\right) \leq -\frac{1}{3}.
$$
Furthermore, we have
$$
\phi_{\tau}(1) - \log \tau \leq (\log\tau) \left(1+1/\tau-1\right) + 2\log(1+1/\tau) = O\left(\frac{\log \tau}{\tau}\right).
$$
Using \eqref{eq:tauDef} once again, it follows that
\begin{align*}
\frac{\left[{r \atop k}\right]}{(r-1)!} \leq \exp\left(\sum_{1 \leq j \leq r-1} \frac{\tau}{\tau+j}\left(\phi_{\tau}(j) - \log \tau\right)\right) \ll \exp\left(-\frac{1}{3} \sum_{1\leq j \leq r-1} \frac{\tau}{\tau+j}\right) = e^{-k/3}.
\end{align*}
Using the fact that $k \geq (\log r)^{1+\delta}$, we deduce that for any $A > 0$ we get
$$
T_2'' \ll N\sum_{k \geq K} e^{-k/3} \ll N e^{-K/3} \ll_A Nr^{-A},
$$
as long as $r$ is large enough (in terms only of $\delta$ and $A$). Adding together all of these contributions, we obtain
$$
\sum_{N/(r+1) < m < N/r} \nu_r(m) \ll \frac{N\log r}{r},
$$ 
which proves the claim. 
\end{proof}

\section*{Acknowledgments}
\noindent We thank Tom Bloom for useful discussions and for finding a small error in the Appendix. We also thank Ofir Gorodetsky and Oleksiy Klurman for helpful advice and encouragement.

\section*{Dedication}
\noindent Cet article est d\'{e}di\'{e} \`{a} ma fille Gila, dont le sourire accueille le jour.

\bibliographystyle{plain}
\bibliography{ShusBib.bib}
\end{document}